\documentclass[final,leqno,letterpaper]{etnanologo}

\usepackage{amsmath}
\usepackage{amssymb}
\usepackage[showonlyrefs]{mathtools}
\usepackage{algorithm}
\usepackage{algpseudocode}
\usepackage{algorithmicx}
\usepackage{graphicx}
\usepackage{mdframed}
\usepackage[utf8]{inputenc}
\usepackage[nameinlink]{cleveref}

\newtheorem{remarksimple}[theorem]{Remark}
\let\oldremarksimple\remarksimple
\renewcommand{\remarksimple}{\oldremarksimple\normalfont}
\newenvironment{remark}{\begin{remarksimple}}{\hfill$\diamond$\end{remarksimple}}

\newtheorem{examplesimple}[theorem]{Example}
\let\oldexamplesimple\examplesimple
\renewcommand{\examplesimple}{\oldexamplesimple\normalfont}
\newenvironment{example}{\begin{examplesimple}}{\hfill$\diamond$\end{examplesimple}}

\newcommand{\sign}{\mbox{sign}}
\newcommand{\mathC}{\mathbb{C}}
\newcommand{\mathR}{\mathbb{R}}
\newcommand{\spann}{\text{span}}
\newcommand{\spec}{\text{spec}}
\newcommand*{\Tra}{\mathsf{T}}
\DeclarePairedDelimiter{\opfences}{(}{)}
\newcommand*{\vect}{\operatorname{vec}\opfences}
\newcommand*{\real}{\operatorname{Re}\opfences}
\newcommand*{\imag}{\operatorname{Im}\opfences}
\newcommand{\af}[1]{#1}
\newcommand{\cut}[1]{}
\newcommand{\lmin}{\lambda_{\min}}
\newcommand{\lmax}{\lambda_{\max}}

\author{Andreas Frommer, Gustavo Ramirez-Hidalgo, Marcel Schweitzer, Manuel Tsolakis\thanks{Department of Mathematics, Bergische Universität Wuppertal, 42097 Wuppertal, Germany (\texttt{\{frommer,ramirezhidalgo,marcel,tsolakis\}@uni-wuppertal.de)}}
}

\title{Polynomial Preconditioning for the Action of the Matrix Square Root and Inverse Square Root}

\crefname{subsection}{section}{sections}
\Crefname{examplesimple}{example}{examples}
\Crefname{example}{Example}{Examples}

\begin{document}

\pagestyle{myheadings} \thispagestyle{plain}
\markboth{A.\ FROMMER, G.\ RAMIREZ-HIDALGO, M.\ SCHWEITZER, M.\ TSOLAKIS}{POLYNOMIAL PRECONDITIONING FOR THE MATRIX SQUARE ROOT}

\maketitle

\begin{abstract} While preconditioning is a long-standing concept to accelerate iterative methods for linear systems, generalizations to matrix functions are still in their infancy. We go a further step in this direction, introducing polynomial preconditioning for Krylov subspace methods which approximate the action of the matrix square root and inverse square root on a vector. Preconditioning reduces the subspace size and therefore avoids the storage problem together with---for non-Hermitian matrices---the increased computational cost per iteration that arises in the unpreconditioned case. Polynomial preconditioning is an attractive alternative to current restarting or sketching approaches since it is simpler and computationally more efficient. We demonstrate this for several numerical examples.
\end{abstract}

\begin{AMS}
    65F60, 
    65F08, 
    65F50, 
    15A16  
\end{AMS}

\section{Introduction} Even if the matrix $A \in \mathC^{n \times n}$ is sparse, its matrix function $f(A) \in \mathC^{n \times n}$ with $f: D\subseteq \mathC \to \mathC$ an appropriate function, is typically a full matrix. This is why for large $A$ one has to resort to computing the action $f(A)b$ for some vector $b$ rather than the full matrix $f(A)$---and in most applications, the action is typically all that is required.

Polynomial or rational Krylov subspace methods are then the methods of choice. Rational Krylov methods require the repeated solution of systems with coefficient matrices $A - \sigma I$ and various shifts $\sigma$. If no efficient solvers are available for such systems, polynomial methods are the only viable approach, and it is these methods that we consider in the present paper.  

Polynomial Krylov subspace methods use a nested orthonormal basis $v_1,\ldots, v_m$ of the Krylov subspaces $\mathcal{K}_m(A,b) = \spann\{b,Ab,\ldots,A^{m-1}b\}$, $m=1,2,\ldots$, from which they then typically extract the \emph{Arnoldi approximation} $f_m$ for $f(A)b$ as 
\[
f_m = V_m \, f(V_m^*AV_m)\, V_m^*b, \enspace \text{ where } V_m = [v_1| \cdots |v_m] \in \mathC^{n \times m}.
\] 
Herein, $V_m^*AV_m =: H_m \in \mathC^{m \times m}$ is an upper Hessenberg matrix, the entries of which arise in the orthogonalization process when determining the basis vectors $v_1,\ldots,v_m$, and $V_m^*b = e_1\|b\|$ with $e_1$ the first canonical unit vector in $\mathC^m$ and $\|\cdot\|$ the Euclidean vector norm.

In the non-Hermitian case, the orthogonalization costs of the Arnoldi process can become prohibitive if a large number of iterations is required. But even in the Hermitian case, there is no short recurrence formula to update $f_{m+1}$ from $f_m$ (except for the linear system case, i.e., $f(z) = z^{-1}$), which means that all Arnoldi vectors must be stored. 

For these reasons, approaches which aim at limiting the dimension $m$ of the Krylov subspace by using \emph{restarts} have attracted a lot of attention in the last two decades~\cite{AfanasjewEtAl2008a,EiermannErnst2006,EiermannErnstGuettel2011,FrommerGuettelSchweitzer2014b,FrommerGuettelSchweitzer2014a,IlicTurnerSimpson2010,TalEzer2007}; see also~\cite{GuettelKressnerLund2020,GuettelSchweitzer2021} (and the references therein) for an overview of limited memory methods for $f(A)b$ in general. The idea is to express the error $f(A)b-f_m$ as the action of a new matrix function $\widehat{f}(A)\widehat{b}$ on a new vector $\widehat{b}$ and recurse. Several representations for $\widehat{f}$ have been studied, and the challenge is to formulate one which is numerically stable. This is possible if $f$ is a Stieltjes transform~\cite{AlzerBerg2002}, a Cauchy type function as defined in~\cite{FrommerGuettelSchweitzer2014a} or a Laplace transform~\cite{FrommerKahlSchweitzerTsolakis2023}. In all these cases, $\widehat{f}$ is given as an integral, and sufficiently precise numerical quadrature has to be employed in order to evaluate $\widehat{f}$ and the corresponding functions occurring after more than one restart. 

Another approach that has recently emerged for improving the performance of Krylov methods for functions of nonsymmetric matrices is to use a \emph{truncated Arnoldi process}, thus mimicking the short recurrence available in the symmetric case. This yields local orthogonality of the basis vectors only, and the approach then combines this with \emph{randomized sketching}~\cite{woodruff2014sketching} which allows to subsequently perform a cheaper (implicit) orthogonalization of the Krylov basis, thus potentially reducing arithmetic operations, storage and communication~\cite{balabanov2019randomized,cortinovis2022speeding,GuettelSchweitzer2023,nakatsukasa2021fast}. These sketching methods can often be applied very successfully, in particular for entire functions like the exponential; see~\cite{PalittaSchweitzerSimoncini2023} for a theoretical justification. For other functions, sketched Krylov methods may fail completely though, as spurious Ritz values outside the field of values of $A$ can cause the methods to essentially break down when they hit a singularity or branch cut of $f$; see~\cite[Section~5.1.3]{cortinovis2022speeding} and~\cite[Section~5.4]{GuettelSchweitzer2023} for examples of this phenomenon.\footnote{We brief\/ly mention that another ``flavor'' of sketched Krylov methods exists, which reduces orthogonalization cost, too, by employing a sketched inner product, possibly together with mixed precision~\cite{balabanov2021randomizedblock,balabanov2022randomized,cortinovis2022speeding,timsit2023randomized}.  Such methods, however, cannot overcome the quadratic dependence of the cost on the iteration number.}

For the functions $f(z) = z^{-1/2}$ and $f(z) = z^{1/2}$ we here propose polynomial preconditioning as an alternative to keep the Krylov dimension $m$---and thus the orthogonalization cost---small. We extract an 
approximation to $f(A)b$ from a Krylov subspace 
$\mathcal{K}_m(Ap_{k-1}(A),c)$, where $p_{k-1}$ is an appropriate 
polynomial of degree $k-1$ and $c$ is a suitably chosen starting vector. While polynomial 
preconditioning as a technique for linear systems and 
eigenvalue computations has been used and investigated 
for a long time~\cite{ashby1988polynomial,dubois1979approximating,freund1990conjugate,joubert1994robust,lanczos1952chebyshev,liu2015polynomial,saad1985practical,saad1987least,thornquist2006fixed}, we believe that its extension to the 
inverse square root and the square root are novel, and our numerical experiments will show that polynomial preconditioning can substantially outperform restarted and sketched approaches. 

Since $\mathcal{K}_{m}
(Ap_{k-1}(A),b) \subseteq \mathcal{K}_{mk}(A,b)$, polynomial 
preconditioning extracts its approximation from a 
smaller subspace than the unpreconditioned method when investing the same 
number of matrix-vector multiplications. In this sense, 
the polynomially preconditioned approximation can be 
expected to be less accurate than the unpreconditioned 
Arnoldi approximation for the same investment of 
matrix-vector multiplications. This is actually a 
theorem for the linear system case and variational 
methods which minimize a measure for the error, like the 
CG or GMRES methods. The possible gains with polynomial 
preconditioning for matrix functions reside in the fact 
that (i) they allow to reduce storage (for Hermitian and non-Hermitian matrices), (ii) they reduce  arithmetic cost due to orthogonalization (for non-Hermitian matrices) which can become prohibitive if many iterations are needed, (iii) they reduce communication cost on parallel machines where orthogonalization requires global communication, (iv) they avoid restarts which are non-trivial to implement for matrix functions, and (v) they can exhibit the typical superlinear convergence behavior of Krylov subspace methods which is usually lost with restarts.

The square root and inverse square root of a matrix have many applications in a variety of scientific computing and engineering applications. For example, the square root and inverse square root of discretized differential operators arise when computing Dirichlet-to-Neumann and Neumann-to-Dirichlet maps~\cite{ArioliLoghin2009,DruskinKnizhnerman1999b}, while fractional powers of the graph Laplacian 
are used for modeling anomalous diffusion and other non-local phenomena in complex networks~\cite{benzi2020non,BenziSimunec2021,Estrada2021}. Applications in data science include sampling from Gaussian Markov Random fields~\cite{pettitt2002conditional}, and whitening to increase the fidelity of stochastic variational Gaussian processes~\cite{pleiss2020fast}, which require inverse square roots of precision and kernel matrices. Due to the relation $\sign(z) = z(z^2)^{-1/2}$,  the inverse square root is also often used in applications where the action of the matrix sign function is needed, e.g., when working with the overlap Dirac operator in lattice quantum chromodynamics~\cite{brannick2016multigrid,VanDenEshofFrommerLippertSchillingVanDerVorst2002}.

To conclude this introduction, let us mention that in the literature two other types of techniques have been proposed which can be regarded as preconditioning for matrix functions in a broader sense. The shift-invert (or RD-rational) Lanczos method~\cite{moret2004rd,eshof2006preconditioning} is a special case of more general rational Krylov subspace methods. It thus requires the solution of a shifted linear system in each iteration, a situation that we consider infeasible in the context of this paper.  
Additionally,~\cite{pleiss2020fast} proposes a preconditioning technique for the (inverse) matrix square root, which alters the resulting vector: The method does not return $f(A)b$, but only a vector which agrees with $f(A)b$ up to certain rotations. While such a vector is all that is needed for the specific application considered in~\cite{pleiss2020fast}, it is not possible to use this kind of approach when the vector $f(A)b$ itself is required (as is typically the case). 

This paper is organized as follows: In \cref{sec:prec} we shortly review left and right preconditioning for linear systems and then show how this can be extended to multiplicative matrix functions and polynomial preconditioners. In \cref{sec:inverse_square_root} we provide algorithmic details of preconditioned methods for the inverse square root and give a theoretical justification for why they can be expected to greatly improve convergence speed. We discuss the extension to the square root in \cref{sec:square_root}. Sensible choices for preconditioning polynomials and how to evaluate them at a matrix argument are considered in \cref{sec:polynomials}. We then report results for several numerical experiments in \cref{sec:numerics} before ending with our conclusions in~\cref{sec:conclusions}.

\section{Polynomial preconditioning} \label{sec:prec}

Preconditioning is a well-established technique for solving linear systems of equations, i.e., when $f(z) = z^{-1}$. For any nonsingular matrix $M$, we have that 
\begin{equation} \label{eq:linsys}
  A^{-1}b \, = \, (M^{-1}A)^{-1}M^{-1}b  \, = \, M^{-1}(AM^{-1})^{-1}b. 
\end{equation}
The first equality in~\eqref{eq:linsys} gives rise to left preconditioning, where we compute approximations $x_m$ to $A^{-1}b$ from the Krylov subspaces $\mathcal{K}_m(M^{-1}A,M^{-1}b)$, and the second equality to right preconditioning, where we obtain approximations $x_m = M^{-1}y_m$ with $y_m$ from the Krylov subspace $\mathcal{K}_m(AM^{-1},b)$. 
The challenge is to find a preconditioner such that $M^{-1}u$ is relatively easy to compute for any vector $u$, and at the same time the preconditioned matrix $AM^{-1}$ or $M^{-1}A$ is close enough to the identity such that a typical Krylov subspace method will take far fewer iterations to converge than the same method using just the matrix $A$.  

The inverse $f(z) = z^{-1}$ has the property that $(z_1z_2)^{-1} = z_1^{-1}z_2^{-1} = z_2^{-1}z_1^{-1}$, and as a matrix function, this translates into the very peculiar property that for any two nonsingular matrices $A$ and $B$ we have
\[
(AB)^{-1} = B^{-1}A^{-1}, \enspace (BA)^{-1} = A^{-1}B^{-1},
\]
which is at the origin of the equalities \eqref{eq:linsys} used for preconditioning.
The order of the factors matters unless $A$ and $B$ commute, in which case $f(A)g(B) = g(B)f(A)$ for any two functions $f$ and $g$ and thus in particular for $f(z) = g(z) = z^{-1}$; cf.~\cite{higham2008functions}.

If we want to transfer the idea of preconditioning to functions $f$ other than $z^{-1}$, one path to follow is to identify situations in which $f(AB)$ can easily be connected to $f(A)$ and/or $f(B)$. The following proposition gives such a result for $f$ a (possibly non-integer) power of a matrix $A$ and $B$ a polynomial in $A$. 

\begin{proposition}\label[proposition]{pro:poly_alpha}
Let $A \in \mathC^{n \times n}$, let $p$ be a polynomial and consider the function $z^\alpha$ for some $\alpha \in \mathR$. If $\alpha < 0$, further assume that the matrices $A$ and $p(A)$ do not have eigenvalues in $(-\infty,0]$. Then
\begin{equation} \label{eq:poly_alpha}
(Ap(A))^\alpha = A^{\alpha} (p(A))^\alpha = (p(A))^\alpha A^{\alpha},
\end{equation}
\end{proposition}
\begin{proof}
By~\cite[Theorem~1.17]{higham2008functions}, we have that if $f(z)=g(h(z))$, then $f(A)=g(h(A))$, assuming that $g$ and $h$ are such that $h(A)$ and $g(h(A))$ are well defined. We apply this result with $g(z) \coloneqq z^\alpha $ and $h(z) \coloneqq zp(z)$, writing $f(z) = g(h(z)) = (z p(z))^\alpha = z^\alpha (p(z))^\alpha$, which immediately gives the first equality in~\eqref{eq:poly_alpha}. The second equality follows because $A^\alpha$ and $(p(A))^\alpha$ commute; see~\cite[Theorem~1.13(a), (e)]{higham2008functions}.
\end{proof}

For the functions considered in \Cref{pro:poly_alpha} we thus have, assuming that $p(A)$ and thus $(p(A))^\alpha$ is nonsingular, that 
\begin{equation} \label{eq:alpha_prec}
A^\alpha b = (Ap(A))^\alpha (p(A))^{-\alpha}b = (p(A))^{-\alpha} (Ap(A))^\alpha b.
\end{equation}
From this we can get, at least in principle, a left polynomially preconditioned method for $A^{\alpha}b$ by extracting iterates $f_m$ as the Arnoldi approximation from the Krylov subspace $\mathcal{K}_m(Ap(A),p(A)^{-\alpha}b)$, and a right polynomially preconditioned method by taking $f_m = (p(A))^{-\alpha}\widetilde{f}_m$ with $\widetilde{f}_m$ the Arnoldi approximations from the Krylov subspace  $\mathcal{K}_m(Ap(A),b)$. However, this becomes a computationally feasible approach only if the action of  $(p(A))^{-\alpha}$ on a vector is easy to compute. This is possible for special choices of $p$ and $\alpha = -1/2$, the inverse square root, and can be extended to the square root, $\alpha = 1/2$, as we will describe in the next sections.

\section{Polynomial preconditioning for the inverse square root} \label{sec:inverse_square_root}
In this section, we first describe algorithms for left and right polynomial preconditioning of the inverse square root. 
We then analyze the effect that preconditioning has on the spectrum (and thus the condition number) of Hermitian positive definite $A$.

In a polynomially preconditioned method for the inverse square root, we want $p(A)$ to approximate $A^{-1}$, so that $Ap(A)$ is close to the identity. At the same time, $(p(A))^{1/2}$ needs to be ``easy'' to evaluate. We thus take $p(z) = (q(z))^2$, where $q$ is chosen as a polynomial that approximates $z^{-1/2}$ to achieve both goals at the same time. The relation \eqref{eq:poly_alpha} becomes, for $\alpha = -1/2$,
\begin{equation} \label{eq:prec_identity}
A^{-1/2} b = (A(q(A))^2)^{-1/2}q(A)b = q(A)(A(q(A))^2)^{-1/2} b,
\end{equation}
provided we have 
\begin{equation} \label{eq:root_of_square}
\left((q(A))^2\right)^{1/2} = q(A).
\end{equation}
Whether \eqref{eq:root_of_square} holds or not depends on the branch that we take for the square root and on the distribution of the eigenvalues of $A$. We will always assume that we take the principal branch of the square root given as
\[
z = |z|e^{i \arg(z)} \mapsto \left||z|^{1/2}\right| e^{i \arg(z)/2}, \enspace \text{ for } \arg(z) \in (-\pi,\pi],
\]
i.e., we put the branch cut on the negative real axis. Then, for any polynomial $q$ we have
$\left((q(z))^2\right)^{1/2} = q(z)$ if and only if $\arg(q(z)) \in (-\pi/2,\pi/2]$, and thus
\begin{equation} \label{eq:root_of_square_A}
    \left((q(A))^2\right)^{1/2} = q(A) \enspace \text{ if } \spec(q(A)) \in \mathC^+,
\end{equation}
with $\mathC^+$ denoting the open right half-plane. We will discuss \eqref{eq:root_of_square_A} again in \cref{sec:polynomials} when we discuss how to obtain appropriate polynomials $q$. 

With $q(A)$ approximating $A^{-1/2}$, the matrix $A(q(A))^2$ should, in a loose sense, be closer to the identity than $A$ is and should thus have a small condition number. This is in turn an indication that we will require fewer Arnoldi iterations. To be specific, if, e.g., $A$ is Hermitian and positive definite, the analysis in~\cite[Theorem~4.3]{FrommerGuettelSchweitzer2014b} shows that the error of the $m$th Arnoldi approximation is bounded by an expression of the form $C \left( \left(\sqrt{\kappa}-1\right) / \left(\sqrt{\kappa} + 1\right) \right)^m$, where $\kappa$ is the condition number of the matrix. See \cref{subsec:conditioning} below for a more in-depth investigation of the condition number of the polynomially preconditioned matrix.

\subsection{Algorithmic aspects}\label{subsec:algorithms}

\begin{algorithm}
\begin{algorithmic}[1]
\State{choose polynomial $q$ such that $q(A)$ approximates $A^{-1/2}$} 
\State{put $c \leftarrow q(A)b$,\enspace $v_1 \leftarrow c/\|c\|$}
\For{$j=1,\ldots,m$} \Comment{Arnoldi process for preconditioned matrix}
\State{compute $u \leftarrow Av_j$, \enspace $y \leftarrow q(A)u, \enspace w \leftarrow q(A)y$} \label{line:polymult}
    \For{$i=1,\ldots,j$}  
    \State{$h_{ij} \leftarrow \langle w, v_i \rangle, \enspace w \leftarrow w - v_ih_{ij}$} \Comment{orthogonalize against previous vectors}
\EndFor
\State{$h_{j+1,j} \leftarrow \|w\|$}
\State{$v_{j+1} \leftarrow w/h_{j+1,j}$}
\EndFor 
\State{$f_m \leftarrow V_m(H_m^{-1/2}e_1\|c\|)$} \Comment{$V = [v_1 | \cdots |v_m]$, $H_m = (h_{ij}) \in \mathC^{m \times m}$ upper Hessenberg} \label{line:fm}
\end{algorithmic}
\caption{$m$ steps of left polynomially preconditioned Arnoldi  for  $A^{-1/2}b$ \label{alg:left_prec_Arnoldi}}
\end{algorithm}

\Cref{alg:left_prec_Arnoldi} describes the left polynomially preconditioned Arnoldi method to approximate $A^{-1/2}b$ in detail. The Arnoldi process produces the orthonormal basis $v_1,\ldots,v_m$ of $\mathcal{K}_m((q(A))^2A,q(A)b)$
obeying the Arnoldi relation
\[
(q(A))^2AV_m = V_{m}H_m + h_{m+1,m}v_{m+1}e_m^\Tra,
\]
with $e_m$ the $m$th canonical unit vector in $\mathC^m$. The vector $H_m^{-1/2}e_1$ for the small matrix $H_m$ is computed using an appropriate method for dense matrices, for example by computing $H_m^{1/2}$ via the blocked Schur algorithm~\cite{deadman2012blocked} and then solving a linear system. In line~\ref{line:polymult} we explicitly stress that we compute the preconditioned matrix-vector product in three stages as multiplications with $A$ and two times $q(A)$. This avoids numerically computing a representation of $p(z) = (q(z))^2$; see \cref{sec:polynomials} for further discussion.

\begin{algorithm}
\begin{algorithmic}[1]
\State{choose polynomial $q$ such that $q(A)$ approximates $A^{-1/2}$} 
\State{put $v_1 \leftarrow b/\|b\|$}
\For{$j=1,\ldots,m$} \Comment{Arnoldi process}
\State{compute $y_j \leftarrow q(A)v_j, \enspace u \leftarrow q(A)y_j, \enspace w \leftarrow Au$} \label{line:polymult_left}
    \For{$i=1,\ldots,j$}  
    \State{$h_{ij} \leftarrow \langle w, v_i \rangle, \enspace w = w - v_ih_{ij}$} \Comment{orthogonalize against previous vectors}
\EndFor
\State{$h_{j+1,j} \leftarrow \|w\|$}
\State{$v_{j+1} \leftarrow w/h_{j+1,j}$}
\EndFor 
\State{$f_m \leftarrow Y_m(H_m^{-1/2} e_1\|b\|)$} \Comment{$Y_m = [y_1 | \cdots |y_m]$, $H_m = (h_{ij}) \in \mathC^{m \times m}$ upper Hessenberg} 
\end{algorithmic}
\caption{$m$ steps of right polynomially preconditioned Arnoldi  for  $A^{-1/2}b$ \label{alg:right_prec_Arnoldi}}
\end{algorithm}

\Cref{alg:right_prec_Arnoldi} gives the details for right preconditioning. We changed the order of factors in line~\ref{line:polymult_left} which allows us to store the preconditioned vectors $y_j = q(A)v_j$ and use them when computing $f_m$. Storing the $y_j$ can be avoided, but then $f_m$ must be computed as $q(A)V_m(H_m^{-1/2}e_1\|b\|)$ which needs an additional (matrix-vector) multiplication with $q(A)$. 

With left preconditioning, the norm $\|f_m\|$ (and by extension, also the norms of the differences $\|f_m-f_{m+k}\|$, $k \geq 1$) can be obtained just from $H_m^{-1/2}e_1\|b\|$ (and $H_{m+k}^{-1/2}e_1\|b\|$), since $V_m$ is orthonormal. This is not the case with right preconditioning, where $Y_m$ does not have orthonormal columns. We conclude that when basing a stopping criterion on the size of the difference of consecutive iterates, left preconditioning is usually more appropriate.   

\subsection{Effect of preconditioning on the spectrum}\label{subsec:conditioning}
To investigate the effect of preconditioning on the spectrum of $A$, we restrict ourselves to Hermitian positive definite $A$, since this is conceptually simpler than more general settings. The preconditioning polynomial $q$ is constructed as an approximation to the inverse square root on the spectral interval of $A$, and clearly the quality of the preconditioner will depend on the quality of this polynomial approximation. As a measure of this quality, we will in the following assume that we have a bound of the form 
\begin{equation}\label{eq:approx_error}
\left|\frac{1}{\sqrt{z}} - q(z)\right| \leq \delta(z) \quad\text{ for } z \in [\lmin,\lmax]
\end{equation}
for the approximation error, where $\lmin$ and $\lmax$ denote the smallest and largest eigenvalue of $A$, respectively. For example, $\delta(z) \equiv \varepsilon$ corresponds to a uniform bound for the \emph{absolute} approximation error on the spectral interval, while $\delta(z) = \varepsilon/\sqrt{z}$ corresponds to a uniform bound for the \emph{relative} approximation error.

If we have a uniform relative error bound available, we can easily find an upper bound for the condition number of the preconditioned matrix $A(q(A))^2$.

\begin{proposition}\label[proposition]{pro:condition}
Let $A$ be Hermitian positive definite with smallest and largest eigenvalue $\lmin$ and $\lmax$, respectively. Further, assume that we have a bound~\eqref{eq:approx_error} with $\delta(z) = \varepsilon/\sqrt{z}$ available, where $\varepsilon < \sqrt{2}-1 \approx 0.4142$. Then
\[
\kappa_{\text{pre}} \leq \frac{1+2\varepsilon+\varepsilon^2}{1-2\varepsilon-\varepsilon^2},
\]
where $\kappa_{\text{pre}}$ denotes the condition number of $A(q(A))^2$.
\end{proposition}
\begin{proof}
By a direct calculation, equation~\eqref{eq:approx_error} implies that
\begin{equation*}
\left|1-z(q(z))^2\right| = \left|z\left(\frac{1}{z}-(q(z))^2\right)\right| = z\left|\frac{1}{\sqrt{z}}-q(z)\right|\left|\frac{1}{\sqrt{z}}+q(z)\right|\leq 2\sqrt{z}\delta(z) +z(\delta(z))^2 
\end{equation*}
for all $z \in [\lmin,\lmax]$. Inserting $\delta(z) = \varepsilon/\sqrt{z}$ gives
\begin{equation}\label{eq:bound_epsilon}
\max_{z \in [\lmin,\lmax]} |1-z(q(z))^2| \leq 2 \varepsilon + \varepsilon^2.
\end{equation}
Due to our assumption that $\varepsilon < \sqrt{2}-1$, the right-hand side of~\eqref{eq:bound_epsilon} is smaller than one, so that $z(q(z))^2$ takes values in $[1-2\varepsilon-\varepsilon^2 , 1+2\varepsilon+\varepsilon^2]$ and the bound on the condition number directly follows.
\end{proof}
\begin{example}
\begin{figure}
\includegraphics[width=0.48\textwidth]{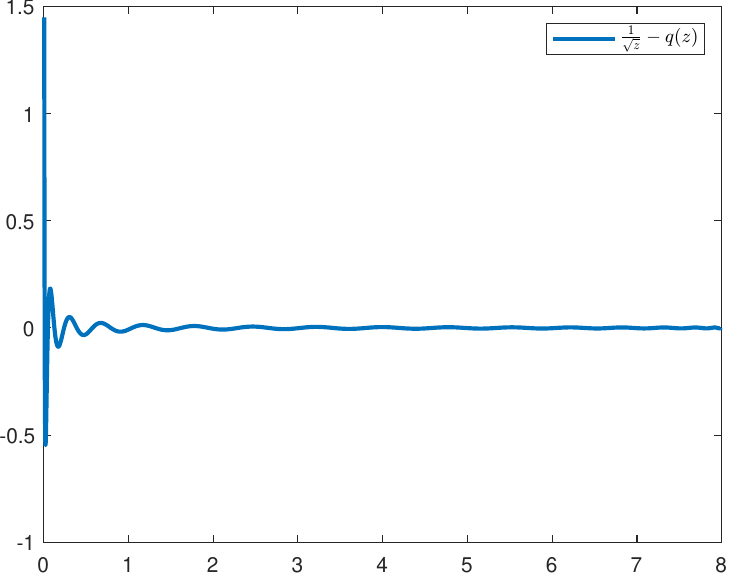}
\hfill
\includegraphics[width=0.48\textwidth]{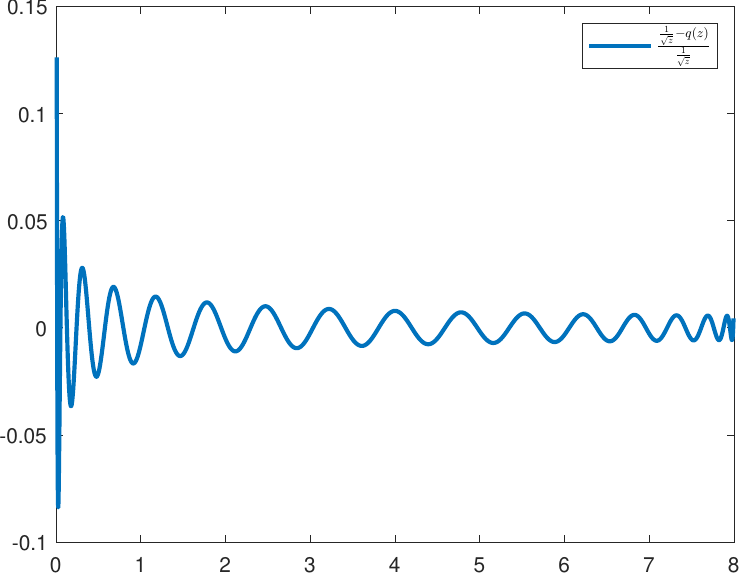}
\\[3ex]
\includegraphics[width=0.48\textwidth]{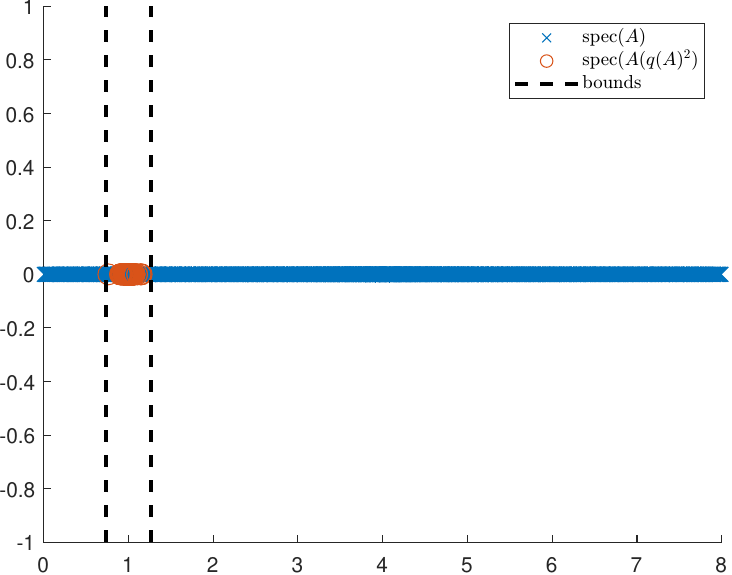}
\hfill
\includegraphics[width=0.48\textwidth]{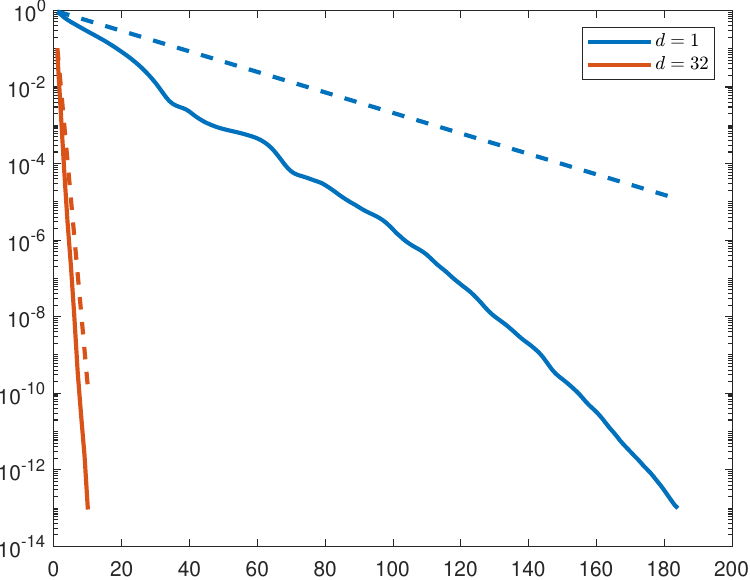}
    \caption{Illustration of effects of polynomial preconditioning for the discretized two-dimensional Laplace operator: Absolute/relative polynomial approximation error on the spectral interval (top left/right), 
    effect on spectrum (bottom left), convergence history and predicted slope (bottom right); see the text for details.}
    \label{fig:prec_illustration}
\end{figure}
We illustrate our theory with a small example. Let $A \in \mathbb{R}^{2500 \times 2500}$ be the discretization of the Laplace operator on a square with 50 interior grid points in each direction, with condition number $\kappa(A) \approx 1054$. We construct a Chebyshev preconditioning polynomial (see \cref{cheb:subseq}) with $d = 32$. 

The top left panel of \Cref{fig:prec_illustration} shows the absolute difference between $q(z)$ and the inverse square root, while the top right panel shows the relative difference. Both in an absolute and in a relative sense, the largest errors occur towards the left end of the spectrum. A uniform relative error bound for this specific example is
\[
\left|\frac{\frac{1}{\sqrt{z}}-q(z)}{\frac{1}{\sqrt{z}}}\right| \leq 0.1263 =: \varepsilon,
\]
so that according to \Cref{pro:condition}, the condition number of $A(q(A))^2$ satisfies $\kappa_{\text{pre}} \leq 1.7345$. In the bottom left panel of \Cref{fig:prec_illustration}, we depict the actual eigenvalues of $A$ and $A(q(A))^2$, together with the bounding interval $[1-2\varepsilon-\varepsilon^2 , 1+2\varepsilon+\varepsilon^2]$. The actual condition number of $A(q(A))^2$ is $1.5153$ and thus even a bit smaller than predicted by the bounds (and roughly a factor 700 smaller than the condition number of $A$). In the bottom right panel of \Cref{fig:prec_illustration}, we show the convergence of the unpreconditioned and preconditioned iteration, together with the estimated convergence slopes based on the condition numbers of $A$ and $A(q(A))^2$, respectively.
\end{example}

\section{Polynomial preconditioning for the square root}\label{sec:square_root}

For the square root, we could, in principle, take the same approach as for its inverse: We let $q$ be a polynomial such that $q(A)$ again approximates $A^{-1/2}$, and then use
\begin{equation}\label{eq:square_root_rational}
A^{1/2}b = (A(q(A))^2)^{1/2}(q^{-1}(A)b) = q^{-1}(A)(A(q(A))^2)^{1/2}b,
\end{equation}
where $q^{-1}(z) = 1/q(z).$ Again, $A(q(A))^2$ will be much better conditioned than $A$ for appropriate choices of $q$, so that the Arnoldi approximations for its square root should converge rapidly. The new aspect is that we now also have to compute the action of $q^{-1}(A)$ on a vector. As the reciprocal of a polynomial, the function $q^{-1}(z)$ is a rational function, and its partial fraction expansion can be determined from the zeros of $q$. Thus, evaluating $q^{-1}(A)b$ means that we have to solve several shifted linear systems with $A$, just what we also have to do in a rational Krylov method. 

\begin{remark}\label{rem:connection_to_rational}
One can make the relation to rational Krylov methods more explicit: A general rational Krylov method extracts its iterates from a space of the form $\pi(A)^{-1}\mathcal{K}_k(A,b)$, where the denominator polynomial $\pi$ is of degree $\leq k-1$; see, e.g.,~\cite{guettel2010rational}. Performing $m$ iterations of a polynomially preconditioned method based on~\eqref{eq:square_root_rational} would correspond to using a denominator polynomial $\pi(z) = (q(z))^m$, i.e., cyclically repeat the same $d$ poles, where $d$ is the degree of $q$. Such a cyclic approach is actually common practice in many rational Krylov methods, as repeated poles allow to reuse matrix factorizations if a direct solver is employed~\cite{beckermann2009error,guettel2010rational}. However, in ``standard'' implementations of rational Krylov methods, the polynomial part of the space is built with the same matrix with which shifted systems are solved. In contrast, a method based on~\eqref{eq:square_root_rational} would build the polynomial part of the space via multiplications with $A(q(A))^2$, i.e., extract its approximation from $(q(A))^{-m}\mathcal{K}_m(A(q(A))^2,b)$. Such a method can thus be regarded as a polynomially preconditioned variant of the standard rational Krylov subspace method. The main computational work in rational Krylov methods lies in linear solves, and due to the typically small iteration numbers, orthogonalization cost and memory requirements are seldom an issue. Therefore, polynomial preconditioning appears to be less appealing in this setting.\hfill$\diamond$
\end{remark}

Since our fundamental assumption was that efficient methods for solving shifted systems are not  available, we consider a different approach instead. The idea is to use the identity $A^{1/2} = A^{-1/2}A$ which gives
\[
A^{1/2}b = A^{-1/2}(Ab),
\]
and then use preconditioning for the inverse square root. Note that the square root is also defined for a singular matrix, provided its eigenvalue 0 is semi-simple, whereas the inverse square root is not. Interestingly, even then we can still proceed by using the (polynomially preconditioned) Arnoldi approximation for the inverse square root on $Ab$, since the eigenvalue~$0$ is effectively deflated from the computation, a phenomenon termed {\em implicit desingularization} in~\cite{BenziSimunec2021}. The basis for this is the following theorem.

\begin{theorem} \label{thm:semi-simple}
Assume that $A \in \mathC^{n \times n}$ is singular and that its eigenvalue 0 is a $k$-fold semi-simple eigenvalue. Let
\[
A = TJT^{-1} \mbox{ with $T$ nonsingular and } J = \oplus_{j=1}^m J_{n_j}(\lambda_j)
\]
be its Jordan canonical form with Jordan blocks
\[
J_{n_j} = \begin{bmatrix} \lambda_j & 1 & 0 & \cdots & 0 \\
                             0     & \lambda_j & 1 & \ddots & 0 \\
                             \vdots &  \ddots  & \ddots & \ddots & \\
                             \vdots &    & \ddots & \lambda_j & 1 \\
                             0  & \cdots & \cdots & 0 & \lambda_j 
                             \end{bmatrix}
                             \in \mathC^{n_j \times n_j}
\]
where $n = \sum_{j=1}^m n_j$. Assume that we order the eigenvalues by enumerating the zero eigenvalues first, i.e.,
\[
\lambda_1 = \ldots = \lambda_k = 0 \mbox{ and } n_1 = \ldots = n_k = 1, \enspace \lambda_j \neq 0  \enspace\text{ for }  j =k+1,\ldots,m.
\]
Denote by $t_j$ the columns of $T =[t_1|\cdots |t_n]$, put $T_{\neg k} = [t_{k+1}|\cdots |t_n]  \in \mathC^{n \times (n-k)} $ and let $V$ be the subspace of $\mathC^n$  spanned by $t_{k+1},\ldots,t_n$. 
Furthermore, 
let $\widehat{A}$ denote the restriction of $A$ onto $V$, defined by
\[
\widehat{A}: V \to V, \quad \widehat{A}t_j = At_j, \qquad j=k+1,\ldots,n.
\]
Then $\widehat{A}$, as a linear map from $V$ to $V$, is nonsingular and 
\[
A^{1/2}y = \widehat{A}^{1/2} y \text{ for all } y \in V.
\]
\end{theorem}
\begin{proof}
As a mapping from $V$ to $V$, the Jordan blocks of $\widehat{A}$ are the blocks $J_{n_j}(\lambda_j)$, $j=k+1,\ldots,m$, and $t_{k+1},\ldots,t_n$ are the corresponding generalized eigenvectors. This shows that $\widehat{A}$ is nonsingular. Moreover, the matrix functions $A^{1/2}$ and $\widehat{A}^{1/2}$ can be characterized using their Jordan canonical forms by their action on the generalized eigenvectors via 
\begin{alignat*}{2}
A^{1/2}T &= T J^{1/2} &&\mbox{ with } J^{1/2} = \oplus_{j=1}^m \big( J_{n_j}(\lambda_j) \big)^{1/2}, \\
\widehat{A}^{1/2} T_{\neg k} &= T_{\neg k} \widehat{J}^{1/2} &&\mbox{ with } \widehat{J}^{1/2} = \oplus_{j=k+1}^m \big( J_{n_j}(\lambda_j) \big)^{1/2}.
\end{alignat*}
So, if $y=  \sum_{j=k+1}^n \eta_j t_j$, then with $a = [ 0 \cdots  0 \; \eta_{k+1}  \cdots \eta_n]^\Tra$ we have
\makeatletter
\tagsleft@false
\makeatother
\begin{align*}
A^{1/2}y &= A^{1/2}Ta \, = \, T J^{1/2} a \, = \, T  \begin{bmatrix} 0_k \\ \oplus_{j=k+1}^m \big( J_{n_j}(\lambda_j) \big)^{1/2} [ \eta_{k+1} \cdots \eta_n]^\Tra  \end{bmatrix} \\
 &=  T_{\neg k} \widehat{J}^{1/2}\begin{bmatrix} \eta_{k+1} \\ \vdots \\ \eta_n \end{bmatrix} = \widehat{A}^{1/2}T_{\neg k}\begin{bmatrix} \eta_{k+1} \\ \vdots \\ \eta_n \end{bmatrix} \, = \, \widehat{A}^{1/2}y.\tag*\endproofhere
\end{align*}
\end{proof}

If $b \in \mathC^n$ is an arbitrary vector and $0$ is a semi-simple eigenvalue of $A$, then the contributions of eigenvectors belonging to this eigenvalue in $y = Ab$ are deflated, i.e., $y$ is of the form assumed in \Cref{thm:semi-simple}. Moreover then, using the notation of \Cref{thm:semi-simple}, we have that for any polynomial $q$ the preconditioned Krylov subspace $\mathcal{K}_m(A(q(A))^2,Ab)$ is the direct sum of $\mathcal{K}_m(\widehat{A}(q(\widehat{A}))^2,\widehat{A}b)$ and $k$ times the space $0$. Practically, this means that we can now choose $q$ such that $q(\widehat{A})$ approximates $\widehat{A}^{-1/2}$ with $\widehat{A}$ nonsingular, and that we can compute Arnoldi approximations for $\widehat{A}^{-1/2}(\widehat{A}b)$ by working with $\mathcal{K}_m(A(q(A))^2,Ab)$ without any modifications at all.    

\section{Choice and evaluation of the polynomial } \label{sec:polynomials}

We now discuss several ways to determine a suitable preconditioning polynomial $q$, i.e., a polynomial of a given degree for which $q(A)$ is a good approximation to $A^{-1/2}$. Of course, the difficulty of this task and its solution depend crucially on properties of the matrix $A$.

An important question is how far we can guarantee or at least expect that the relation \eqref{eq:root_of_square} is fulfilled for the polynomial that we choose, i.e., whether we have
$((q(A))^2)^{1/2} = q(A)$. In this respect, we can state the following general result. 

\begin{theorem}  \label{thm:approx:_quality}
Assume $\spec(A) \subseteq \mathC^+$ and that $q$ approximates $z^{-1/2}$ on $\spec(A)$ uniformly  in a relative sense with accuracy $\tfrac{1}{\sqrt{2}}$, i.e., we have
\begin{equation} \label{eq:relative_accuracy_sqrt2}
|q(\lambda) - \lambda^{-1/2}| \leq \frac{1}{\sqrt{2}} |\lambda^{-1/2}| \enspace \text{ for } \lambda \in \spec(A).
\end{equation}
Then $((q(A))^2)^{1/2} = q(A)$. 
\end{theorem}
\begin{proof} Since $\spec(A) \subseteq \mathC^+$, we have $|\arg(\lambda^{-1/2})| \leq \tfrac{\pi}{4}$ for all $\lambda \in \spec(A)$. Because of \eqref{eq:relative_accuracy_sqrt2}, this implies
\[
|\arg(q(\lambda))| < \frac{\pi}{2},
\]
which in turn gives $|\arg((q(\lambda))^2)|  < \pi$ and thus $((q(\lambda))^2)^{1/2} = q(\lambda)$ for all $\lambda \in \spec(A)$, i.e., $((q(A))^2)^{1/2} = q(A)$.
\end{proof}

The assumption on the relative approximation error in \eqref{eq:relative_accuracy_sqrt2} is 
quite restrictive.  Heuristically, it seems justified to assume that for small values of $\lambda \in \mathC^+$, a polynomial $q$ that approximates the inverse square root is such 
that $q(\lambda) \in \mathC^+$, too, since, after all, $\lambda^{-1/2} \in \mathC^+$ is large with $| \arg(\lambda^{1/2})| \leq \tfrac{\pi}{4}$. For large values of $\lambda$ though, it might happen 
that $q(\lambda) \in \mathC^-$. If this is the case, we do not have $((q(A))^2)^{1/2} = q(A)$, meaning that in the left preconditioned method, e.g., we are starting with the ``wrong'' vector $c = 
q(A)b$ rather than the correct $\widehat{c} = ((q(A))^2)^{1/2}b$. The expansions of $c$ and $\widehat{c}$ in terms of the generalized eigenvectors of $A$ differ in those eigenvectors 
that belong to the large eigenvalues $\lambda$ for which $q(\lambda) \in \mathC^-$.  This difference won't affect our computation of $A^{-1/2}c$ too much, since the inverse square root effectively damps the components belonging to eigenvectors with large eigenvalues. 

To summarize this discussion, we see that it is not trivial to guarantee that we have  $((q(A))^2)^{1/2} = q(A)$. In special situations, we might be able to actually assert this with mathematical certainty. In many other cases, we might have strong heuristics or numerical indications that this identity holds. 
One such situation is when $\spec(A) \subset \mathC^+$, where then $1/\sqrt{\lambda}$ is in the wedge $\{z \in \mathC^+, |\arg(z)| \leq \pi/4\}$, so that we can expect $q(\lambda) \in \mathC^+$ at least for all eigenvalues $\lambda$ which are not too large. Moreover, if we do not aim for particularly high accuracy, we can accept that for some of the large eigenvalues $\lambda$ we violate $((q(\lambda))^2)^{1/2} = q(\lambda)$ when computing the action of the inverse square root. In any case, i.e., even when $((q(A))^2)^{1/2} \neq q(A)$, left or right polynomial preconditioning always yields an inverse square root which, though, might be non-primary: With left preconditioning we obtain $f = Bb$ with  $B = (A(q(A))^2)^{-1/2}((q(A))^2)^{1/2}$ which satisfies $B^2 = A^{-1}$, and similarly for right preconditioning. We refer to \cite{higham2008functions} for further considerations on non-primary, primary, and principal (inverse) square roots.       
\subsection{Chebyshev expansions} \label{cheb:subseq}
A real function $f: [a,b] \subseteq \mathR \to \mathR$ which is absolutely integrable with respect to the weight function $w(x) = ((x-a)(b-x))^{-1/2}$ gives rise to a truncated Chebyshev series expansion
\begin{equation} \label{eq:cheb_expansion}
\sum_{i=0}^k c_iT^{[a,b]}_i(z),
\end{equation}
where $T_i^{[a,b]}$ is the (scaled) Chebyshev polynomial of the first kind of degree $i$ for the interval $[a,b]$, and the coefficients $c_i$ are obtained as
\[
c_0 = \frac{1}{\pi} \int_a^b w(z) \cdot f(z)T_0^{[a,b]}(z) dz, \enspace\enspace  c_i = \frac{2}{\pi}\int_a^b w(z) \cdot f(z)T_i^{[a,b]}(z) dz, \; i= 1,2,\ldots \, .
\]

So, if the spectrum of $A$ is enclosed in the real interval $[a,b] \subset (0, \infty)$---which is the case in particular if $A$ is Hermitian and positive definite---we can, for a given degree, take $q$ as the corresponding truncated Chebyshev expansion for $z^{-1/2}$ on $[a,b]$ . The theory and the practical implementation of Chebyshev expansions are well studied and understood; see, e.g.,~\cite{boyd2001chebyshev}. In particular, we can use  a straightforward matrix-vector version of the Clenshaw recurrence~\cite{Clenshaw1955,FoxParker1968} to accurately compute the action of $q(A)$ on a vector $b$.

If $\spec(A) \subseteq [a,b]$ with $a > 0$, the relation $((q(A))^2)^{1/2} = q(A)$ which we need in order to obtain the principal square root (see \eqref{eq:root_of_square}), is 
fulfilled for the truncated Chebyshev expansion polynomial $q$ if $q$ is positive on $[a,b]$. We can check this at least numerically by evaluating the polynomial on a discrete set of points in $[a,b]$. 

\subsection{Polynomials interpolating at (harmonic) Ritz values}\label{subsec:ritz}
Another way to obtain a preconditioning polynomial is by using the Arnoldi method itself. The Ritz values, i.e., the eigenvalues of the upper Hessenberg matrix $H_{d}$ arising from $d$ steps of the Arnoldi process, can be used as approximations for the eigenvalues of $A$ (and they tend to first approximate exterior eigenvalues well); see, e.g.~\cite[Section~6]{saad2011numerical}. 

It is therefore a natural idea to choose the preconditioning polynomial $q$ in our method as the polynomial of degree $d-1$ that interpolates $1/\sqrt{z}$ at the $d$ Ritz values. This has the attractive feature that it does not require any {\em a priori} knowledge about the spectral region of $A$ but rather adapts itself automatically to the spectrum of $A$. This approach comes at the cost of $d$ additional Arnoldi steps, i.e., $d$ matrix-vector products and $d(d-1)/2$ inner products. The Arnoldi process for constructing $H_{d}$ can either be started with the vector $b$, or with a randomly drawn vector, which might sometimes be preferable; see~\cite[Section~3.6]{loe2022toward} for a discussion of this topic in the context of polynomial preconditioning for linear systems.

Alternatively, one might want to use \emph{harmonic Ritz values} as interpolation points. 
These are the inverses of the Ritz values of $A^{-1}$ with respect to the space $A\mathcal{K}_d(A,b)$ and can be computed as the eigenvalues of
\[
\widetilde{H}_d := H_d + h_{d+1,d}^2H_d^{-*}e_de_d^*.
\]
The inverse $\mu^{-1}$ of a harmonic Ritz value $\mu$ is contained in the field of values $\mathcal{F}(A^{-1}) \subset \mathC$ of $A^{-1}$~\cite[Section~5.1]{sleijpen2000jacobi},
and thus $|\mu| > 1/\max\{|z|: z \in \mathcal{F}(A^{-1})\}$. This shows that for a nonsingular matrix $A$ the harmonic Ritz values are bounded away from $0$, while standard Ritz values can become $0$ (or arbitrarily close to $0$) if $0 \in \mathcal{F}(A)$. It is this property which often makes harmonic Ritz values more attractive than standard Ritz values when the function has a pole at $0$. Using preconditioning polynomials based on (harmonic) Ritz value information has recently received quite a lot of attention in the context of solving linear systems~\cite{liu2015polynomial,loe2020phd,loe2022toward,loe2020polynomial} and eigenvalue computations~\cite{embree2021polynomial}.

If the field of values 
$\mathcal{F}(A)$ of $A$ is in $\mathC^+$, so are its spectrum and its standard as well 
as its harmonic Ritz values. So we know that the constructed polynomial $q$ satisfies $((q(\mu))^2)^{1/2} = q(\mu) = \mu^{-1/2}$ for the standard or harmonic Ritz values $\mu$, and we can  interpret this as an indication that we may expect that indeed $((q(\lambda))^2)^{1/2} = q(\lambda)$ for $\lambda \in \spec(A)$.  

\af{When working with interpolating polynomials, care has to be taken to use a representation which is favorable numerically. Following \cite{NacReiTre92} and \cite{liu2015polynomial,loe2020phd,loe2022toward,loe2020polynomial}, we represent in our numerical experiments the polynomial in its Newton form based on a Lej\`a ordering of the Ritz values and evaluate it with a Horner-type scheme.  }

\subsection{Polynomials obtained via error minimization}\label{subsec:ls_polynomials}
A more involved way to obtain a preconditioning polynomial is adapting the strategy presented in~\cite{ye2021} for GMRES. 
Let $\Gamma$ be a contour, and assume that its interior contains the spectrum of $A$ or at least the major part of it. Such $\Gamma$ may be available due to {\em a priori} information or it can be constructed as a polygon from the Ritz values of the Arnoldi process; see~\cite{ye2021}. 
Note that $A$ allows for a unique (the ``principal'') inverse square root $A^{-1/2}$ if it has no eigenvalues on $(-\infty,0]$. Thus, we can further assume that $\Gamma$ and its interior exclude $(-\infty,0]$.

If we discretize $\Gamma$ into a set of points $\omega = \{z_1, \dots, z_N\}$, then an inner product for the space $\mathcal{P}_{d-1}$ of polynomials of degree at most $d-1<N$ is given by
\begin{equation}\label{eq:polynomial_inner_product}
    \langle p_1, p_2 \rangle_\omega = \sum_{i=1}^N p_1(z_i) \overline{p_2(z_i)}
\end{equation}
with corresponding norm $\|\cdot\|_\omega$. 
We obtain the preconditioning polynomial $q(z)$ by solving the least-squares problem (where we extend $\| \cdot \|_\omega$ to functions other than polynomials)
\begin{equation*}
    \min_{q\in\mathcal{P}_{d-1}} \|z^{-1/2} - q(z)\|_\omega.
\end{equation*}

The theoretical justification for this approach is that if the discretization is fine enough, the absolute value of the approximation error $z^{-1/2}-q(z)$ is small on all of the boundary $\Gamma$. 
As this error is holomorphic (note that we excluded $(-\infty,0]$ from the interior of $\Gamma$), the maximum modulus principle~\cite[Theorem~10.24]{Rudin} states that $|z^{-1/2} - q(z)|$ cannot attain a larger value in the interior of $\Gamma$. Since $\Gamma$ is constructed such that it approximately encloses the eigenvalues of $A$, $q(z)$ should thus yield a good approximation of $z^{-1/2}$ for all eigenvalues of $A$.

Because of the pole of $z^{-1/2}$ at $0$, it becomes increasingly difficult to acquire a polynomial with small approximation error the closer $\Gamma$ is to $0$. Thus, it can be helpful to require a minimum absolute value of $\Gamma$. We implement this by replacing the part of $\Gamma$ closer to $0$ than a user-specified minimum distance by the circular arc of corresponding radius.

As described in~\cite{ye2021}, the above choice for a polynomial inner product allows us to use the standard Arnoldi process to construct an implicit polynomial basis. For this, we represent a polynomial $p$ by the unique vector $\vect{p}$ containing its values on $z_1,\dots,z_N$, i.e., $\vect{p}=[p(z_1), \dots, p(z_N)]^\Tra$. Then the inner product \eqref{eq:polynomial_inner_product} corresponds to the standard Euclidean inner product on these vector representations.
If we apply the Arnoldi process to the matrix $\operatorname{diag}(z_1,\dots,z_N)$ and the vector $[1,\dots,1]^\Tra/N^{1/2} \in\mathR^N$, then the $k$th Arnoldi vector represents the $k$th orthonormal basis polynomial $p_{k-1}(z)$ (with respect to \eqref{eq:polynomial_inner_product}) of degree $k-1$. We express $q(z)$ in terms of these basis polynomials, so the least-squares problem
\begin{equation*}
    \min_{q\in\mathcal{P}_{d-1}} \|z^{-1/2} - q(z)\|_\omega = \min_{\alpha_j \in \mathC} \|z^{-1/2} - \sum_{j=0}^{d-1} \alpha_j p_j(z)\|_\omega
\end{equation*} 
is equivalent to 
\begin{equation}\label{eq:ls_problem}
    \min_{\alpha_j \in \mathC} \| \vect{z^{-1/2}} - \sum_{j=0}^{d-1} \alpha_j \vect{p_j} \|_2 = \min_{\alpha \in \mathC^{d}} \| \vect{z^{-1/2}} - P_{d} \alpha \|_2,
\end{equation}
where the matrix $P_{d}=[\vect{p_0} | \dots | \vect{p_{d-1}}]\in\mathC^{N\times d}$ contains the Arnoldi vectors as its columns. Since they are orthonormal, the vector $\alpha$ with components $\alpha_i$  is obtained as $\alpha = P_{d}^* \vect{z^{-1/2}}$.
Once $\alpha$ is known, we can evaluate $q(A)v$ using an Arnoldi-like process that does not require knowledge of $P_{d}$, see~\cite[Section~3]{ye2021} for more information. This process involves $d-1$ matrix-vector products with $A$ and $\mathcal{O}(d^2)$ vector operations, but no inner products.

Within this approach, we know the value of $q(z)$ on discrete points on the contour $\Gamma$ which (approximately) contains the spectrum of $A$. If then $q(z) \in \mathC^+$ for these points, this can be taken as an indication that indeed $q(\lambda) \in \mathC^+$ for all $\lambda \in \spec(A)$ as is required in \eqref{eq:root_of_square}: If $q(z) \in \mathC^+$, i.e., $\real{q(z)} > 0$ not only for the discrete points on the contour $\Gamma$, but on all of the contour, then $\real{q(z)} > 0$ for all $z$ inside the contour, since the real part of the holomorphic function $q$ is harmonic and thus attains its minimum for all $z$ inside $\Gamma$ on the boundary $\Gamma$, see Section~15.1 and in particular Theorem~15.1g in \cite{henrici3}. 

We note that we could actually add $\real{\vect{q}}\geq 0$ as a constraint to \eqref{eq:ls_problem}. This constraint can be written as
\begin{equation*}
  0 \leq 
  \begin{bmatrix}
    I & 0
  \end{bmatrix}
  \begin{bmatrix}
    \real{\vect{q}}\\ \imag{\vect{q}}
  \end{bmatrix} 
  =
  \begin{bmatrix}
    I & 0
  \end{bmatrix}
  \begin{bmatrix}
    \real{P_{d+1}} & -\imag{P_{d+1}} \\
    \imag{P_{d+1}} & \real{P_{d+1}}
  \end{bmatrix}
  \begin{bmatrix}
    \real{\alpha} \\ \imag{\alpha}
  \end{bmatrix}.
\end{equation*}
Furthermore, we can rewrite \eqref{eq:ls_problem} in a similar manner, separating real and imaginary components, so that in total we obtain a real inequality-constrained linear least-squares problem. These are known to be equivalent to convex quadratic programming problems, for which several efficient algorithms exist; see, e.g., \cite[Section~5.2.2]{bjorck} and references therein. In the following examples, this will not be necessary, however, as the constraint will be satisfied without having been required explicitly in the optimization problem \eqref{eq:ls_problem}.

\section{Numerical examples} \label{sec:numerics}
We now illustrate the benefits of polynomial preconditioning for several examples. The experiments in \Cref{ex:poisson,ex:graph_laplace} are run in MATLAB R2022a on a computer with Intel Core i7-1185G7 8-core CPU (3.0 GHz) and 32 GB RAM under Ubuntu 20.04, while \Cref{ex:qcd} has been implemented in parallel via \texttt{C} and \texttt{MPI} and was run on the JUWELS Cluster of the J{\"u}lich Supercomputing Centre. Each node of JUWELS consists of two Intel Xeon Platinum 8168 CPUs running at 2.7 GHz and with 96 GB of RAM; see \cite{juwels} for further information.

\subsection{Example~\arabic{subsection}: Inverse square root for the 3d discrete Laplace operator} \label[example]{ex:poisson} 
In our first example we take $A \in \mathbb{R}^{n \times n}, n = 10^6$, as the discrete, three-dimensional Laplace operator on a grid with $100 \times 100 \times 100 $ equidistant interior points. The spectral interval of $A$ is $[6(1-\cos\bigl(\tfrac{\pi}{N+1}\bigr)),\, 12-6(1-\cos\bigl(\tfrac{\pi}{N+1}\bigr))]$, and we obtained preconditioning polynomials via the Chebyshev expansion over the spectral interval as described in \cref{cheb:subseq}. 
We (numerically) checked that $q_{d-1}(z)$ for the various degrees $d-1$ used does not have a zero on the spectral interval $[a,b]$ and thus maps $[a,b]$ onto a positive real interval. This means that by  \eqref{eq:root_of_square_A} we indeed have $((q(A))^2)^{1/2} = q(A)$.

The top row of \Cref{fig:poisson3d} depicts the relative 2-norm $\|f_m-A^{-1/2}b\| / \|A^{-1/2}b\|$ of the error of the Arnoldi approximation $f_m$ as a function of the iteration number $m$; see \Cref{alg:right_prec_Arnoldi}. The ``exact'' solution $A^{-1/2}b$ that we use to compute the errors is actually an approximation computed using a restarted Arnoldi iteration with guaranteed error bounds as described in~\cite{frommer2016error} with the bound set to $10^{-14}$.

The left plot in \Cref{fig:poisson3d} is without preconditioning, the right plot is for $d=8$, i.e., we use a degree 7 polynomial as a right preconditioner. Both plots also show the relative $2$-norm $\|f_{m+k}-f_{m}\|/\|f_{m+k}\|$ of the difference of two iterates (as dashed lines), a quantity that is easily available also for large matrices---even without explicitly forming the iterates. Here, the parameter $k \geq 1$ controls how frequently this error estimate is evaluated (and thus, how frequently the inverse square root of the respective $m \times m$ Hessenberg matrix needs to be computed). The plots show that these norms of the differences start to stagnate at the same time when the errors start to stagnate, so this observation can be used to devise a simple stopping criterion. To account for the fact that iterations of the preconditioned method are much more costly (in particular for higher values of $d$) than iterations of the plain method, but one expects that much fewer iterations are required for convergence, we make the parameter $k$ smaller the larger $d$ is. Specifically, in this experiment with $d = 1, 2, 4, \dots, 64$, we choose $k = 64/d$, so that we check the residual norm every 64th iteration in the plain method and every iteration when $d = 64$.

\begin{figure}
\includegraphics[width=0.48\textwidth]{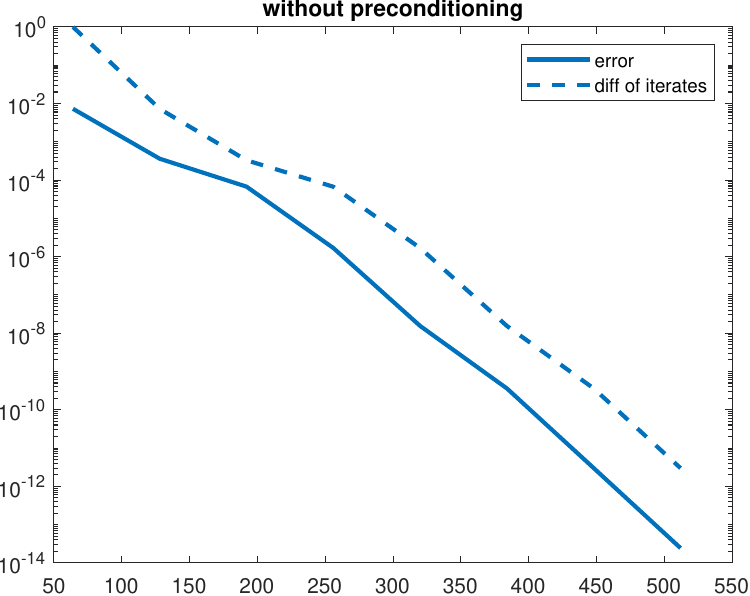}
\hfill
\includegraphics[width=0.48\textwidth]{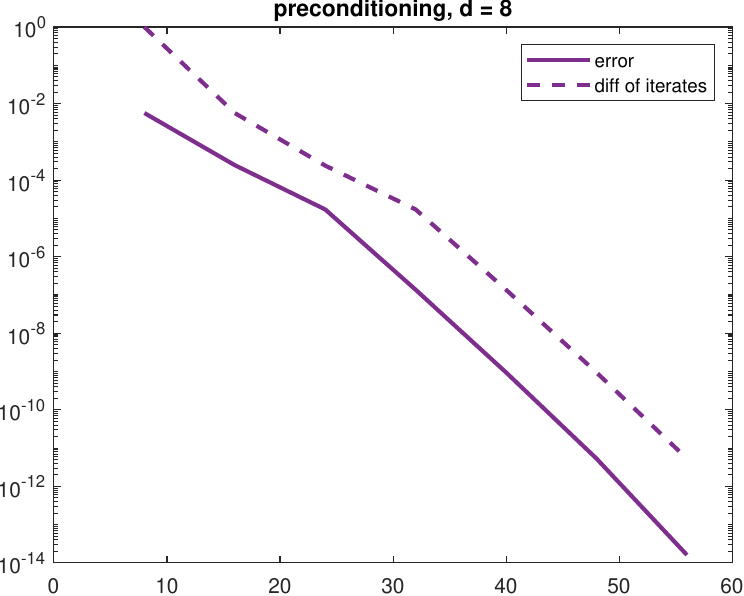}
\\[3ex]
\includegraphics[width=0.48\textwidth]{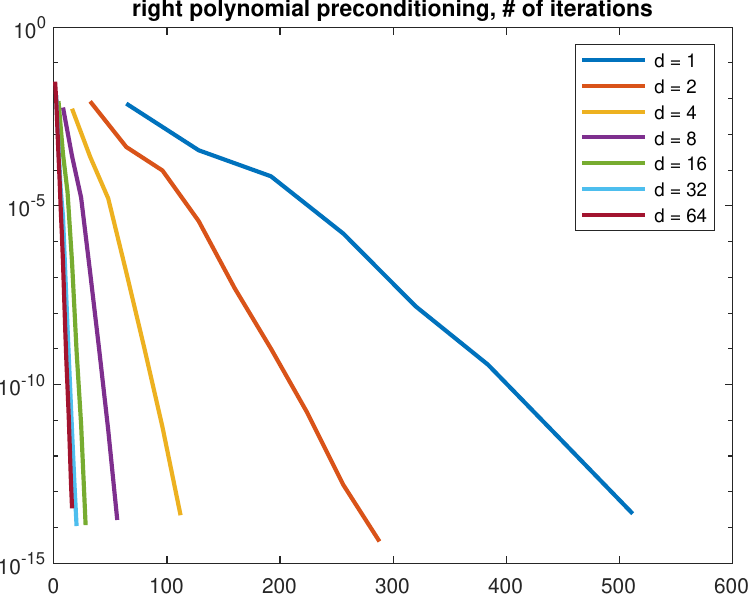}
\hfill
\includegraphics[width=0.48\textwidth]{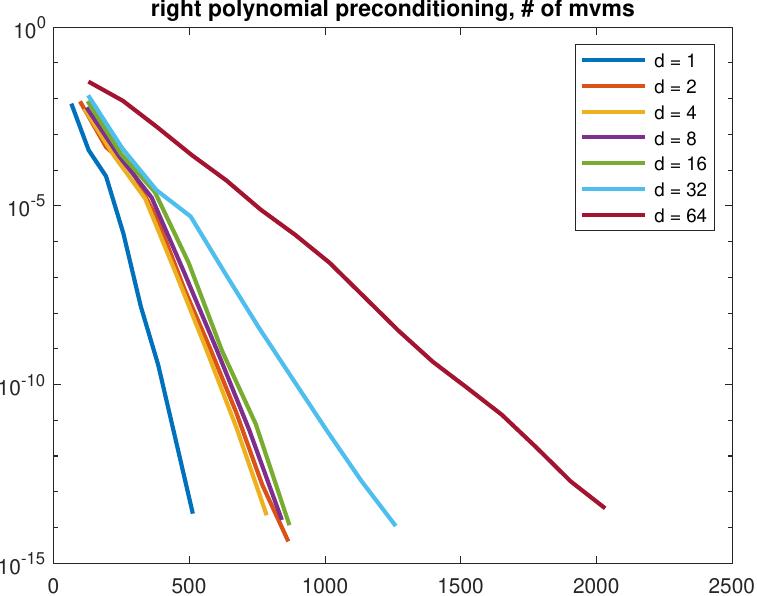}
\caption{Relative error when approximating $A^{-1/2}b$ with Chebyshev preconditioning polynomials of various degrees $d-1$ ($d = 1$ corresponds to an unpreconditioned method), where $A$ is the discretization of the three-dimensional Laplace operator and $b$ is a random vector of unit norm.} \label{fig:poisson3d}
\end{figure}

\begin{table}
\caption{Timings and operation counts for approximating $A^{-1/2}b$ with Chebyshev preconditioning polynomials of various degrees, where $A$ is the discretization of the three-dimensional Laplace operator and $b$ is a random vector of unit norm. The polynomial $q$ has degree $d-1$, i.e., the first row corresponds to the plain, unpreconditioned Lanczos method. Run times are given both for the standard (1P) and two-pass (2P) version of the algorithm. \label{tab:poisson3d}}
\begin{center}
\scalebox{.9}{%
\begin{tabular}{c|ccccc} 
$d$ & iterations & mvms & inner products & time 1P (in s) & time 2P (in s)  \\ \hline
1  &  512 &  512  & 1\,024 & 8.0   & 14.5 \\
2   &  288 &  864  & 576 & 12.9   & 25.8 \\
4   &  112 &  784  & 224  & 10.7   & 20.4 \\
8   &  56  &  840  & 112  & 12.2   & 20.7 \\
16  &  28  &  868  & 56  & 10.8  & 22.8 \\ 
32  &  20  &  1\,260  &  40  & 16.9  & 32.6 \\ 
64  &  16  &  2\,032  &  32  & 24.3  & 53.9
 \end{tabular}}
\end{center}    
\end{table}

The bottom part of \Cref{fig:poisson3d} shows results for various degrees $d-1$ of the preconditioning polynomial. The left plot depicts the error as a function of the number of Arnoldi iterations, while the right plot shows it as a function of the number of matrix-vector multiplications (mvms). Recall that preconditioning with a degree $d-1$ polynomial requires $2d-1$ matrix-vector multiplications per iteration. We can take this plot as an information about the computational cost, since mvms are the by far most costly operations here. We observe that the preconditioned Arnoldi methods take increasingly more mvms as $d$ is increased, and they all take more mvms than the standard method. However, if storage were an issue and we could not store the 512 Arnoldi vectors, the standard method could not be performed as such but rather as a two-pass Lanczos approach, where the first part does not store the basis vectors but just assembles the tridiagonal matrix $H_m$ (see, e.g.,~\cite{frommer2008matrix,GuettelSchweitzer2021}). Once the coefficient vector $H_m^{-1/2}e_1\|b\|$ is computed, a second pass is done which then combines the anew computed basis vectors using these coefficients. In this way, the number of mvms is actually twice as much---1\,024 in our example---as what we see in the plot. Then, if for example we can store the about 110 vectors needed for preconditioning with $d=4$, the preconditioned method with its 784 mvms takes 24\% less mvms than the standard method. 

In \Cref{tab:poisson3d} we report the wall clock time for running the plain and preconditioned Lanczos method, both in a one pass and two pass version, together with the number of iterations, mvms and inner products needed to reach a relative error norm below $10^{-12}$. Among all tested methods, plain (unpreconditioned) Lanczos has the smallest run time, followed by polynomial preconditioning with $d = 4$. As explained in the previous paragraph, for very large scale problems, it might actually be relevant to compare plain \emph{two}-pass Lanczos with that of preconditioned \emph{one}-pass Lanczos, as the latter method builds a much smaller subspace and is thus less likely to require two passes due to memory constraints. We therefore note that for all values of $d$ between $2$ and $16$, the run time of the one-pass preconditioned method is lower than that of the two-pass unpreconditioned method.

\subsection{Example~\arabic{subsection}: Sign function for the overlap operator in lattice QCD} \label[example]{ex:qcd}
For a square matrix $A$, its sign function $\sign(A)$ can be expressed as $A(A^2)^{-1/2}$, so that the computational burden in computing $\sign(A)b$ resides in computing $(A^2)^{-1/2}b$. Note that $A^2$ will not be computed explicitly; we rather compute $A^2x$ for a given vector $x$ as two consecutive mvms with $A$. 

The sign function of a large, non-Hermitian matrix arises in the overlap operator in lattice Quantum Chromodynamics (QCD) as we now shortly explain.  
QCD is the fundamental physical theory of the quarks and gluons as the constituents of matter, and their interaction via the strong force is described by the Dirac operator. 
The Wilson-Dirac operator $D_{w}$ arises as a discretization of the Dirac operator on a finite four-dimensional Cartesian lattice. The Wilson-Dirac operator acts on discrete spinor fields which have twelve components per grid point, corresponding to all possible combinations of three color and four spin indices \cite{gattringer2009quantum}. 

Now, $D_{ov}(\mu)$, the {\em overlap Dirac operator} at chemical potential $\mu$, preserves chiral symmetry, an important physical property, on the lattice while other discretizations as, e.g., $D_w$ do not. To be specific, the overlap Dirac operator takes the form~\cite{bloch2006overlap,hernandez1999locality,Neuberger1998}
\begin{equation*}\label{eq:overlap_operator}
    D_{ov}(\mu) = I + \rho \Gamma_{5} \sign{( \underbrace{\Gamma_{5} D_{w}(m_{w},\mu) )}_{=: Q(m_w,\mu)}}.
\end{equation*}
Here, $\Gamma_5$ is a simple diagonal matrix which acts as the identity on spinor components belonging to spins 1 and 2 and as the negative identity on  those belonging to spins 3 and 4, and $\rho \in (0,1)$ is a mass parameter, typically close to $1$.
In the argument of the sign function, $D_{w}(m_{w},\mu)$ is the massless Dirac-Wilson operator with an appropriately chosen shift $m_{w} \in (-2,0)$ and a chemical potential of $\mu$. It is the presence of $\mu \neq 0$ that makes $Q(m_{w},\mu)$ non-Hermitian; see \cite{bloch2006overlap}. For notational simplicity, we abbreviate $Q(m_w,\mu)$ as $Q_\mu$ from now on. 
   
In our computations, we use a Dirac matrix for a grid with dimensions $64{\times}32^{3}$ coming from a physically relevant ensemble provided by the lattice QCD group at the University of Regensburg via the Collaborative Research Centre SFB-TRR55, with parameters $m_{0} = -0.332159624413$ and $c_{sw} = 1.9192$ \cite{bali2014moment}. Hence $Q_\mu$ has $25\,165\,824$ rows and columns. We took $m_{w} = -1.4$ and $\mu=0.3$ which are physically relevant values. All our results relate to the evaluation of $(Q_\mu^2)^{-1/2}b$. In an actual simulation in lattice QCD one has to solve systems with $D_{ov}(\mu)$ repeatedly, using an iterative solver. Each iteration then requires an evaluation of $\sign(Q_\mu)b$ for some vector $b$, and this in turn is obtained by computing $(Q_\mu^2)^{-1/2}b$.
 
We obtain the preconditioning polynomial $q$ via interpolation at the Ritz values (of $Q_\mu^2$) as outlined in \cref{subsec:ritz}. Since all Ritz values were always contained in the right half-plane, so were the values of $q(z)= z^{-1/2}$ at these Ritz values, which is important in view of \eqref{eq:root_of_square}. When using harmonic Ritz values instead of standard Ritz values, we consistently needed about $10\%$ more iterations across all tested values of $d$. For this reason, we do not report results using harmonic Ritz values.

Our numerical experiments were run on 64 and 256 nodes on the JUWELS Cluster and we use 2 \texttt{MPI} processes per node and 24 \texttt{OpenMP} threads per process.\footnote{code available at \url{https://github.com/Gustavroot/sign_function_LQCD_with_polyprec}} 
As the ``exact" solution $f^{*}$, we took the unpreconditioned $f_{m+k}$ with $m=6\,000$ and $k=64$ as before, and we obtained the error measure $\|f_{m}-f_{m+k}\|_{2}/\|f_{m+k}\|_{2} \approx 3.0\cdot 10^{-10}$.

Computing $H_{m}^{-1/2}$ for the Hessenberg matrices $H_{m}$ is done using the SLEPc library \cite{slepc}, with its cost in seconds as a function of $m$ shown in the right plot of \Cref{fig:overlap_qcd}.
The time to obtain $H_{m}^{-1/2}$ was around five times larger than that for running the whole Arnoldi with $m=6\,000$ when running on 64 nodes on JUWELS, rendering the computation of $\|f_{m+k}-f_{m}\|_{2}/\|f_{m+k}\|_{2}$ ten times more expensive than the time for Arnoldi for that value of $m$.
These times appear to be overly large, and we are convinced that they could be reduced substantially by looking at the details of the implementation. Rather than doing that, though, we decided to simply exclude the time needed for checking the stopping criterion from all the timings reported in \Cref{tab:overlapQCD}. Note that the unpreconditioned method would have to check the stopping criterion more often and evaluate the sign function on larger Hessenberg matrices, so excluding this from the timings is in favor of the unpreconditioned method.

The unpreconditioned case converges quite slowly so that running the Arnoldi method with full orthogonalization up to $m=6\,000$ is certainly not realistic in practice because of the very high orthogonalization cost.  For our comparison with the preconditioned methods, we therefore reduced the iteration number of unpreconditioned Arnoldi to $1\,600$ iterations. This resulted in a relative error of around $4.0 \cdot 10^{-5}$ at $m=1\,600$, which we then used as
the stopping criterion in the preconditioned runs, too. 
The left part of \Cref{fig:overlap_qcd} shows the relative error as a function of the iteration counts and \Cref{tab:overlapQCD} gives operations counts and timings. We checked the error every $k=64/d$ iterations.

\begin{figure}
\includegraphics[width=0.48\textwidth]{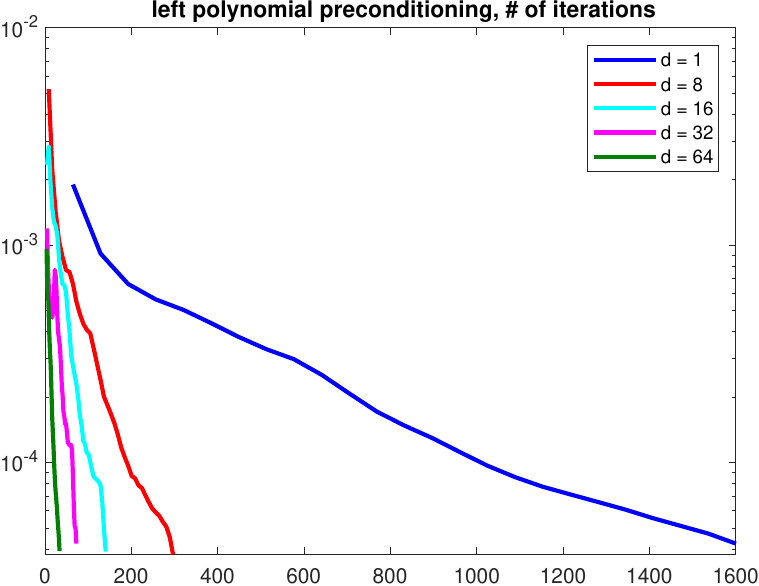}
\hfill
\includegraphics[width=0.48\textwidth]{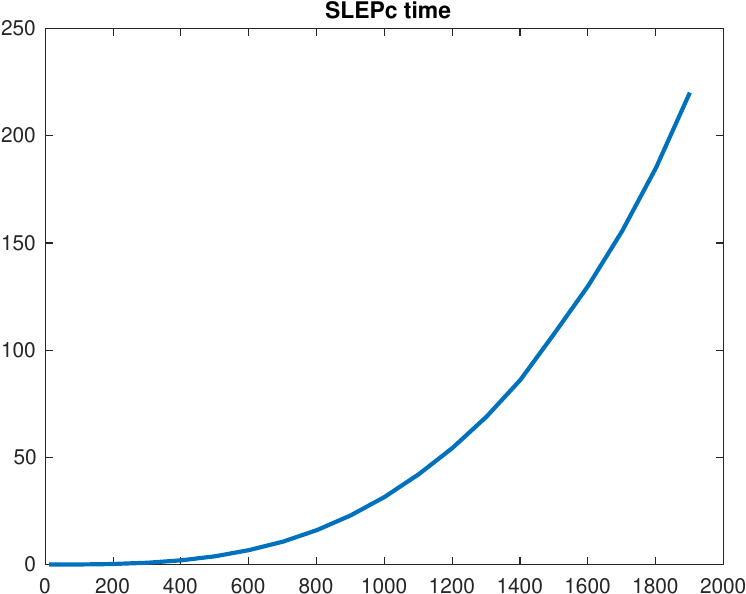}
\caption{Results for approximating $(Q_\mu)^{2})^{-1/2}b$ with Arnoldi preconditioning polynomials of various degrees, $b$ a random vector. Left: relative error as a function of the Arnoldi basis size up to a value of $1\,600$. Right: time (in s) to compute $H_{m}^{-1/2}$ as a function of $m$ with SLEPc.} \label{fig:overlap_qcd}
\end{figure}

\begin{table}
\caption{Timings and operation counts for approximating $ ((Q_\mu)^{2})^{-1/2}b$ with Arnoldi preconditioning polynomials $q$ of various degrees $d-1$, $b$ a random vector. The polynomial $q$ has degree $d-1$, i.e., the first row corresponds to the plain, unpreconditioned Arnoldi method. \label{tab:overlapQCD}}    
\begin{center}
\begin{tabular}{c|ccccc}
$d$ & iterations & mvms & inner products & time 64 nodes (in s) & time 256 nodes (in s)  \\ 
\hline
 1  &   1\,600 &  3\,200    &  1\,279\,200  &  127.8  &  105.8  \\
 8   &  296    &  8\,910    &  42\,510      &  25.7   &  9.8    \\
 16  &  140    &  8\,742    &  9\,991       &  12.3   &  7.8    \\
 32  &  72     &  9\,198    &  3\,125       &  11.6   &  7.4    \\
 64  &  33     &  8\,636    &  2\,578       &  10.6   &  5.5
\end{tabular}
\end{center}
\end{table}

From \Cref{fig:overlap_qcd} and \Cref{tab:overlapQCD} we see that, as in the previous example, polynomial preconditioning reduces the number of iterations while increasing the total number of mvms. Still, the reduction in execution time is very pronounced for the preconditioned methods, because orthogonalization costs dominate when a large number of iterations have to be performed. The fastest preconditioned method arises for $d=64$,  and it is by a factor of twelve faster than the unpreconditioned method when using 64 nodes, and by a factor of 19 on 256 nodes. In the additional \Cref{fig:overlap_qcd_high_tol}, we show convergence plots for the best three preconditioned methods from \Cref{tab:overlapQCD}, now going down to a relative error of $10^{-9}$, a precision which is very difficult to attain for the unpreconditioned algorithm due to tremendous time and memory requirements. For this higher accuracy, $d=32$ and $d=64$ take almost the same time, while  $d=16$ takes about $10\%$ more time since the orthogonalization cost, which grows quadratically in the number of iterations, starts to prevail. The case $d=64$ with $m=94$ is 135 times faster than the unpreconditioned one with $m=6\,000$ on 256 nodes.

\begin{figure}
\centering\includegraphics[width=8cm]{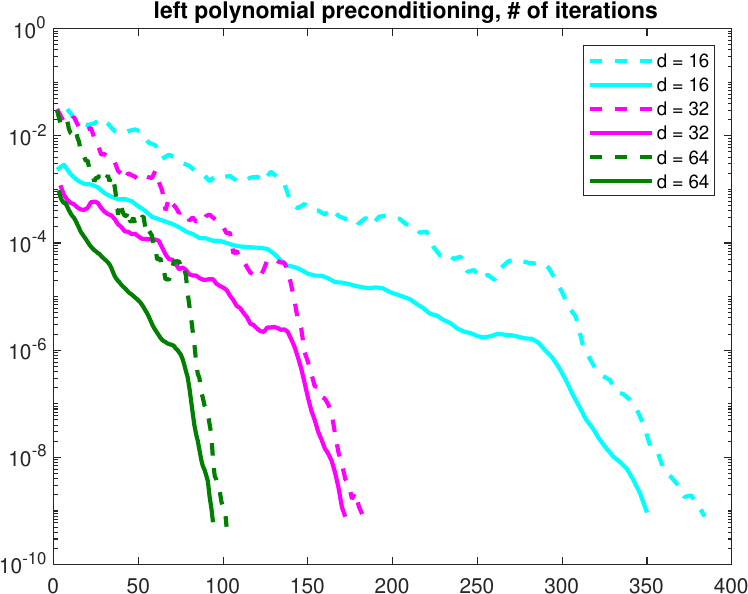}
\caption{Results for approximating $((Q_\mu)^{2})^{-1/2}b$ with Arnoldi preconditioning polynomials of various degrees, $b$ a random vector. Solid lines: relative error $\|f_m-f^*\|/\|f^*\|$, dashed lines: error measure $\|f_m-f_{m+k}\|/\|f_{m+k}\}$ with $k = 64/d$. %
} \label{fig:overlap_qcd_high_tol}
\end{figure}

\subsection{Example~\arabic{subsection}: Square root of graph Laplacian} \label[example]{ex:graph_laplace}
As our last example we study another nonsymmetric problem, which is known to be notorious\-ly difficult to solve both for sketched and for restarted Krylov subspace methods; see, e.g.,~\cite[Section~5.1.3]{cortinovis2022speeding} and~\cite[Section~5.4]{GuettelSchweitzer2023} where a similar (but much smaller) model problem is considered. We let $A$ be the adjacency matrix of a directed graph, $D_{\text{in}}$ the diagonal matrix which contains the in-degrees of all nodes on the diagonal and $L = D_{\text{in}}-A$ the corresponding \emph{in-degree Laplacian}.

The square root $L^{1/2}$---a {\em fractional Laplacian}---is an important tool for modeling non-local dynamics on the network; see, e.g.,~\cite{benzi2020non,metzler2000random}. It is evident from its definition that all column sums of $L$ are zero and it is thus singular. Its spectrum is contained in the closed right half-plane as can be seen via Gershgorin's theorem, while its field of values is not; see~\cite[Theorem~1.6.6]{horn1991topics}).

We are interested in forming $L^{1/2}b$, where $b$ is a randomly chosen canonical unit vector $e_i$. As outlined in \cref{sec:square_root}, we apply our method to approximate $L^{-1/2}(Lb)$, which is possible despite $L$ being singular; cf.~\Cref{thm:semi-simple}. To increase numerical stability, we run the Arnoldi method with reorthogonalization in this example. The graph we took is
\texttt{Kamvar/Stanford} from the Suite\-Sparse matrix collection \cite{suitesparse}, so $L$ has $281\,903$ rows and columns and $2\,594\,228$ nonzero entries.

\begin{figure}
\includegraphics[width=0.48\textwidth]{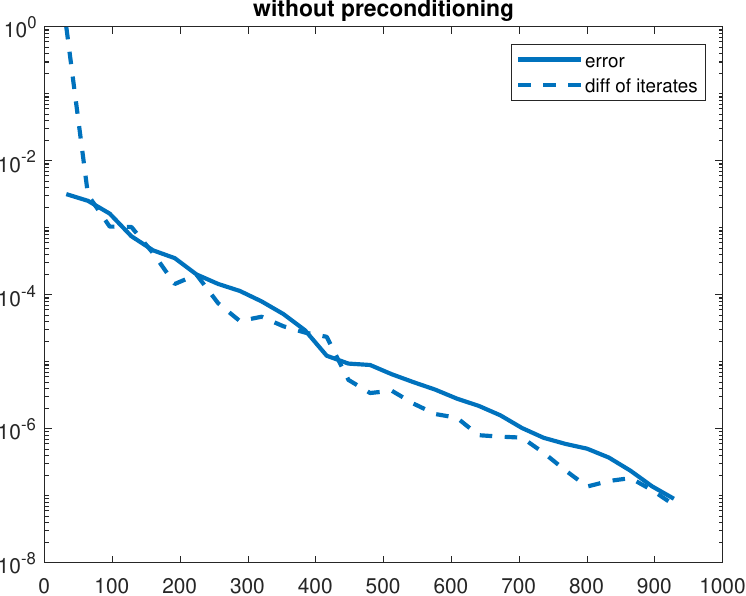}
\hfill
\includegraphics[width=0.48\textwidth]{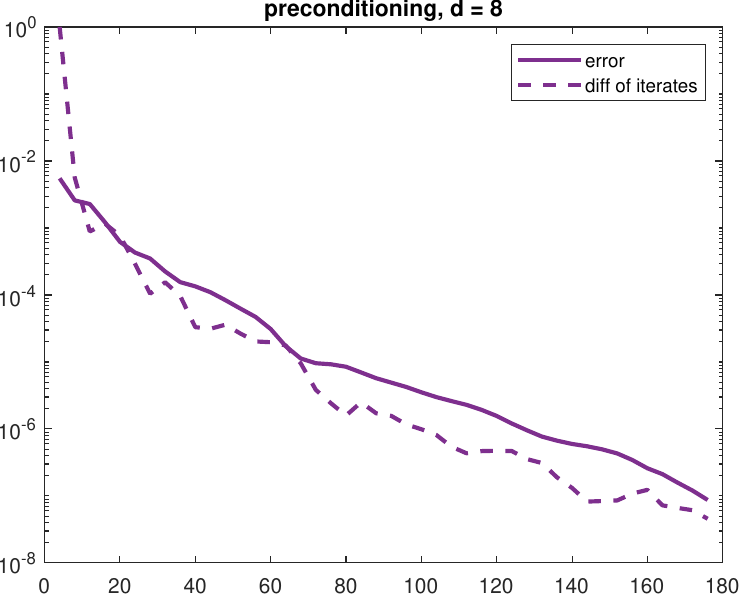}
\\[3ex]
\includegraphics[width=0.48\textwidth]{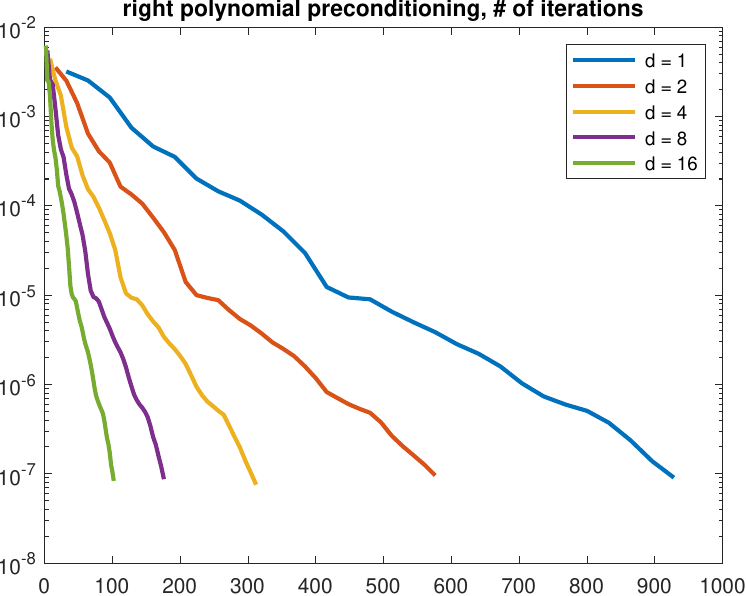}
\hfill
\includegraphics[width=0.48\textwidth]{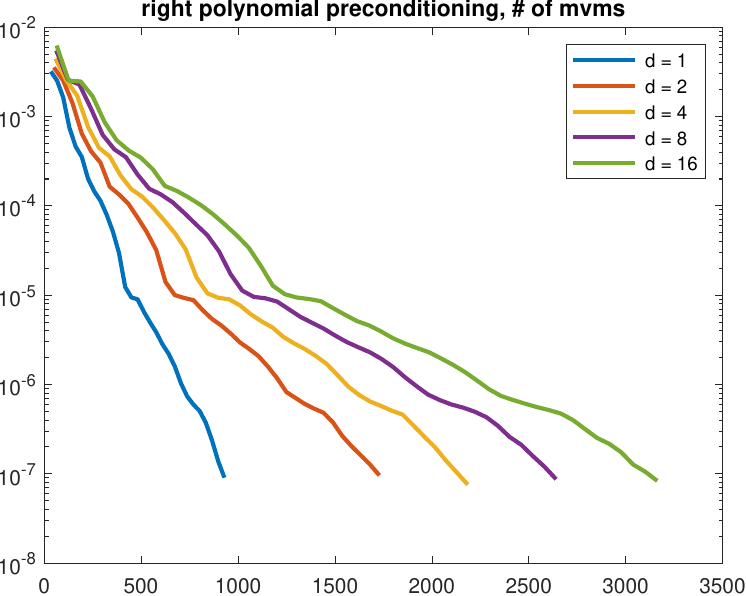}
\caption{Results for approximating $L^{1/2}b$ with Arnoldi preconditioning polynomials of various degrees ($d = 1$ corresponds to an unpreconditioned method), where $L$ is the graph Laplacian of the network \texttt{Kamvar/Stanford} and $b$ is a random vector of unit norm.} \label{fig:graphlaplace}
\end{figure}

\begin{table}
\caption{Timings and operation counts for approximating $L^{1/2}b$ with two choices for the preconditioning polynomials of various degrees $d-1$. Left: polynomials interpolate in Ritz values (see \cref{subsec:ritz}). Right: error minimizing polynomials (see \cref{subsec:ls_polynomials}. $L$ is the graph Laplacian of the network \texttt{Kamvar/Stanford}. The first row corresponds to unpreconditioned Arnoldi.  \label{tab:graphlaplace}}    
\begin{center}
\begin{tabular}{c|cccc|cccc}
\multicolumn{1}{c}{} & \multicolumn{4}{c}{interpolation of Ritz values} &\multicolumn{4}{c}{error minimization} \\ 
$d$ & iters & mvms & inprods & time (in s)& iters & mvms & inprods & time (in s)\\ 
\hline
  1  &   928  &  929    &  860\,256 &  333.0 &  
         928  &  929    &  860\,256 &  333.0\\
  2  &   576  &  1\,732  &  331\,202 &  141.7 & 
         544 &  1\,693 &  298\,932 &  393.9\\  
  4  &   312  &  2\,190  &  97\,044 &  58.1 & 
         272 &  1\,965 &   77\,252 &  125.3  \\  
  8  &   176  &  2\,650  &   30\,856 &  37.7 &  
         144 &  2\,221 &   24\,132 &   67.4 \\  
 16  &   102  &  3\,180  &   10\,542 &  35.1 &
         78 &  2\,479 &    9\,546 &   60.7   
 \end{tabular}
\end{center}
\end{table}

In \Cref{fig:graphlaplace} and the left part of \Cref{tab:graphlaplace} we report results obtained with preconditioning polynomials interpolating at Ritz values (see \cref{subsec:ritz}), this time aiming for a relative error of $10^{-7}$.
Given the large problem size, as in \Cref{ex:poisson} we take as the ``exact'' solution, which we use to compute the errors, the approximation obtained with the unpreconditioned Arnoldi method with a stricter tolerance, $10^{-10}$. The parameter $k$ that controls how frequently we check the stopping criterion is chosen as $k = 32/d$ in this experiment. As in the previous example, the benefits of polynomial preconditioning become very apparent: As orthogonalization cost largely dominates the overall cost of the unpreconditioned method, run time is reduced by a factor of about $9$ when going from unpreconditioned Arnoldi to polynomially preconditioned Arnoldi with $d = 8$ or $d = 16$.
Note that \Cref{tab:graphlaplace} now reports complete timings, including the time spent in evaluating $H_m^{-1/2}b$ (via the function \texttt{sqrtm} of MATLAB followed by a linear system solve). As opposed to \Cref{ex:qcd}, the time spent in these evaluations now represents only a small fraction (from 0.5\% for $d = 16$ to 2.8\% for $d = 1$) of the total compute time. 

The right part of Table~\ref{tab:graphlaplace} shows results obtained with preconditioning polynomials solving the least-squares problem discussed in \cref{subsec:ls_polynomials}. 
For this, we run $60$ Arnoldi iterations and use the MATLAB function \texttt{boundary} on the Ritz values to construct the boundary $\Gamma$. We choose a minimum absolute value of $0.1$ for all of $\Gamma$ and discretize the modified boundary in steps of uniform size $0.005$. In view of \eqref{eq:root_of_square} we checked whether for the discrete points $z$ on the boundary we have $q(z) \in \mathC^+$, which was indeed always the case.
We refrain from showing convergence plots here, as they do not give additional insight beyond what we see in \Cref{fig:graphlaplace} for the Arnoldi polynomial.

 The least-squares polynomials are complex and thus necessitate complex arithmetic even though $L$ and $b$ are real. This is in contrast to the interpolating polynomials which are real because all Ritz values are. As a consequence, with least squares polynomials the run time increases for $d=2$ compared to the unpreconditioned case and remains larger than the run times with the interpolation polynomials for the same values of $d$. However, fewer matrix-vector and inner products are needed, so we expect this approach to be more beneficial in cases where complex arithmetic is needed anyways.

To conclude this example, we compare to the performance of competing methods that can be used to reduce storage and orthogonalization cost when approximating functions of nonsymmetric matrices. Specifically, we compare with the sketched FOM method from~\cite{GuettelSchweitzer2023}\footnote{available at~\url{https://github.com/MarcelSchweitzer/sketched_fAb}.} as well as with quadrature based restarting~\cite{FrommerGuettelSchweitzer2014a},\footnote{available at \url{https://github.com/guettel/funm_quad}.} both with and without implicit deflation~\cite{EiermannErnstGuettel2011}. The results are depicted in \Cref{fig:graphlaplace_competitors}. 

For restarted Arnoldi, we vary the restart lengths from $20$ to $200$ and allow at most $3000$ mvms overall. Without implicit deflation, the method only manages to reach the desired accuracy of $10^{-7}$ for the largest restart length $k = 200$, within nine restart cycles. This requires $1\,800$ mvms as well as $179\,100$ inner products and takes $91.8$ seconds. 

When using implicit deflation, we aim to deflate $k/10$ approximate eigenvectors (which is typically a suitable value). This speeds up convergence, particularly for $k = 50$ and $k = 100$. The overall most efficient method is then for $k = 100$, taking $60.9$ seconds.

\begin{figure}
\centering 
\includegraphics[width=0.48\textwidth]{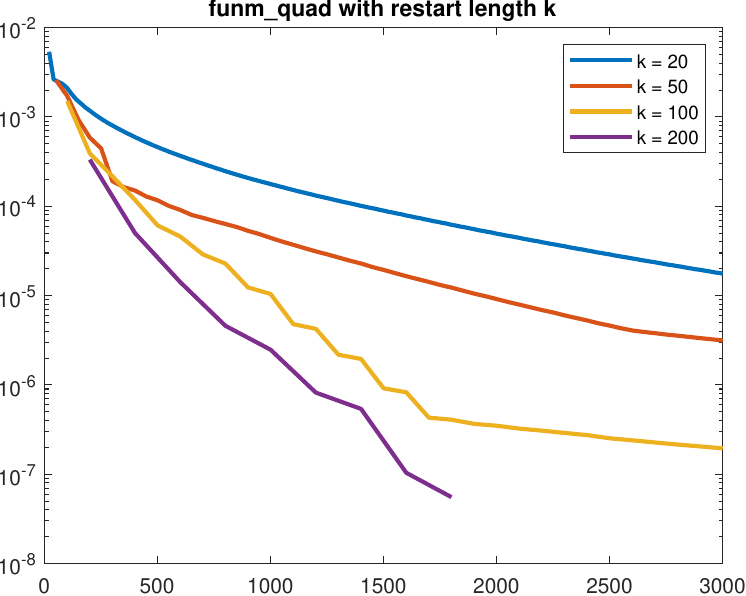}
\hfill
\includegraphics[width=0.48\textwidth]{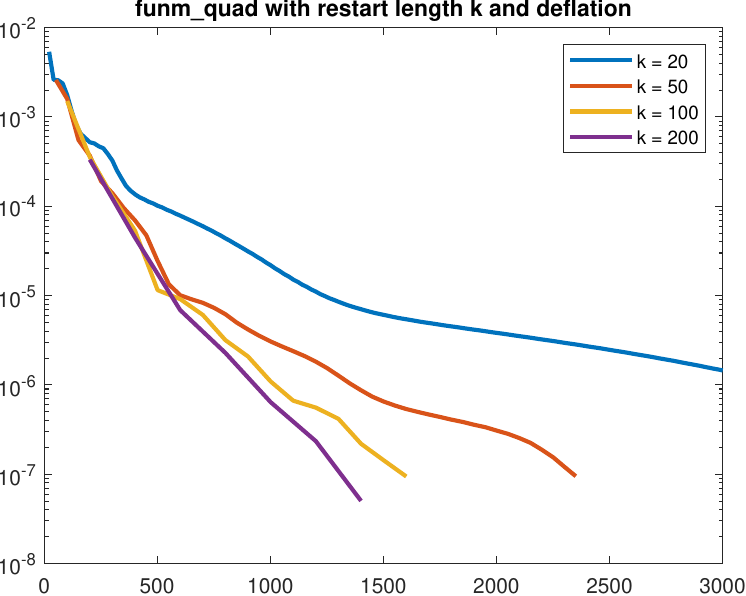}
\\[3ex]
\includegraphics[width=0.48\textwidth]{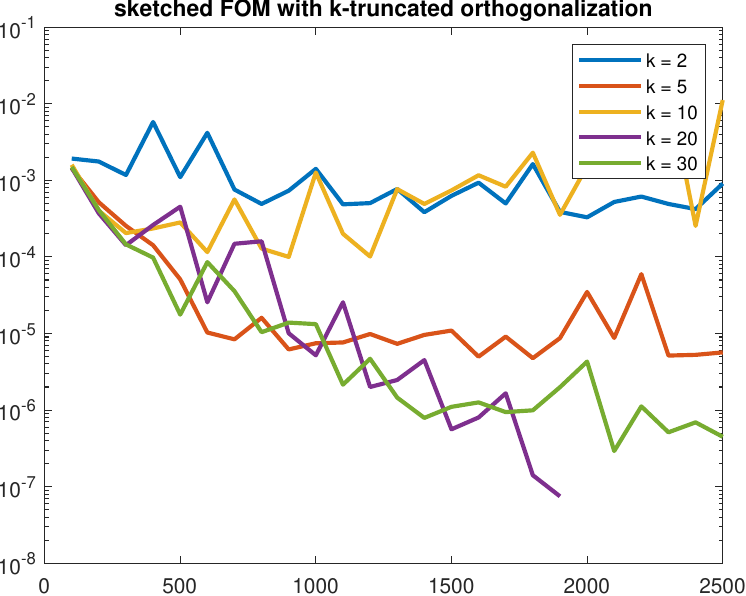}
\caption{Results for approximating $L^{1/2}b$ by sketched-and-truncated or restarted Krylov methods with various truncation/restart lengths $k$, where $L$ is the graph Laplacian of the network \texttt{Kamvar/Stanford} and $b$ is a random vector of unit norm.} \label{fig:graphlaplace_competitors}
\end{figure}

\begin{table}
\caption{Timings and operation counts for approximating $L^{1/2}b$ by sketched-and-truncated or restarted Krylov methods, where $L$ is the graph Laplacian of the network \texttt{Kamvar/Stanford} and $b$ is a random vector of unit norm. For each class of methods, only the value of $k$ leading to the most efficient outcome is reported in this table. \label{tab:graphlaplace_competitors}}   
\begin{center}
 \begin{tabular}{l|cccc}
method & trunc./rest.\ length $k$ & mvms & inner products & time (in s)  \\ \hline
sketched FOM  & 20 & 1\,900 & 37\,790 & 79.9 \\
restarted Arnoldi  & 200 & 1\,800 & 179\,100 & 91.8 \\
deflated restarted Arnoldi & 100 & 1\,400 & 69\,300 & 60.9 
 \end{tabular}
\end{center} 
\end{table}

In sketched FOM, the Krylov basis is computed by a $k$-truncated orthogonalization, i.e., each new basis vector is only orthogonalized against the $k$ preceding basis vectors (so that $k = 2$ mimics the Lanczos process). We try different truncation parameters ranging from $2$ to $30$. Several observations can be made: For the smaller truncation lengths, the method stagnates or diverges long before reaching the target accuracy. In general, convergence is very irregular and erratic, which also means that error estimates and stopping criteria can become unreliable. Additionally, the practical choice of a good truncation parameter appears to be non-trivial, as there seems to be no systematic dependence of performance on the truncation parameter (e.g., $k=5$ performs better than $k=10$ and $k=20$ performs better than $k=30$). The method performs best for $k = 20$ and manages to reach the desired accuracy within $1\,900$ iterations, requiring $79.9$ seconds.

Overall, preconditioned Arnoldi with $d=8$ or $d=16$ outperforms the most efficient other method by a factor of roughly 2 in wall clock time.

\section{Concluding remarks}\label{sec:conclusions}
We presented a way to use polynomial preconditioning in the Arnoldi method for computing  the action of the inverse square root of a matrix on a vector, and by expansion, also for the square root itself. Due to the reduction of the number of Arnoldi steps, the polynomially preconditioned method saves on orthogonalization and storage cost, and these savings can be very substantial, particularly in the non-Hermitian case. Our numerical examples show that polynomial preconditioning may outperform other approaches which aim at avoiding long recurrences, too. We also discussed why it is important to take into account that the square root has two branches.

\paragraph{Acknowledgment}
The authors gratefully acknowledge the Gauss Centre for Supercomputing e.\,V.\ (\url{www.gauss-centre.eu}) for funding this project by providing computing time through the John von Neumann Institute for Computing (NIC) on the GCS Supercomputer JUWELS at Jülich Supercomputing Centre (JSC), under the project with id MUL-TRA. This work is partially supported by the German Research Foundation (DFG) research unit FOR5269 “Future methods for studying confined gluons in QCD”.

\bibliographystyle{siam}
\bibliography{lit}

\end{document}